\documentclass[11pt]{article}%
\usepackage{amsfonts}
\usepackage{amsmath}
\usepackage{amssymb}
\usepackage{graphicx}

\providecommand{\U}[1]{\protect\rule{.1in}{.1in}}
\newtheorem{theorem}{Theorem}

\newtheorem{conjecture}[theorem]{Conjecture}
\newtheorem{corollary}[theorem]{Corollary}

\newtheorem{definition}[theorem]{Definition}

\newtheorem{lemma}[theorem]{Lemma}

\newtheorem{proposition}[theorem]{Proposition}
\newtheorem{remark}[theorem]{Remark}

\newenvironment{proof}[1][Proof]{\noindent\textbf{#1.} }{\ \rule{0.5em}{0.5em}}

\def\Z{\mathbb{Z}}

\def\Q{\mathbb{Q}}
\newcommand{\ds}{\displaystyle}
\newcommand{\ol}{\overline}

\newcommand{\on}{\operatorname}

\begin{document}

\title{\textbf{Computing relative power integral bases in a family of quartic
extensions of imaginary quadratic fields}}
\author{\textsc{Zrinka Franu\v{s}i\'{c} and Borka Jadrijevi\'{c}}}
\date{}
\maketitle

\begin{abstract}
Let $M$ be an imaginary quadratic field with the ring of integers $\mathbb{Z}_{M}$
and let $\xi$ be a root of polynomial%
\[
f\left(  x\right)  =x^{4}-2cx^{3}+2x^{2}+2cx+1,
\]
where $c\in\mathbb{Z}_{M},$ $c\notin\left\{  0,\pm2\right\}  .$ We consider an 
infinite family of octic fields $K_{c}=M\left(  \xi\right)$ with the ring of
integers $\mathbb{Z}_{K_{c}}.$ Our goal is to determine all generators of
relative power integral basis of $\mathcal{O=}\mathbb{Z}_{M}\left[
\xi\right]  $ over $\mathbb{Z}_{M}.$ We show that our problem reduces to
solving the system of relative Pellian equations%
\[
cV^{2}-\left(  c+2\right)  U^{2}=-2\mu\text{, \ \ }cZ^{2}-\left(  c-2\right)
U^{2}=2\mu\text{, }%
\]
where $\mu$ is an unit in $\mathbb{Z}_{M}$. We solve the system completely and
find that all non-equivalent generators of power integral basis of
$\mathcal{O}$ over $\mathbb{Z}_{M}$ are given by $\alpha=\xi,$ $2\xi-2c\xi
^{2}+\xi^{3}$ for $\left\vert c\right\vert \geq159108$ and $|c|\leq200$.

\end{abstract}

\footnotetext{
\par
The authors where supported by the Croatian Science Foundation under the
project no. 6422.}

\section{Introduction}

\noindent Let $K$ be an algebraic number field of degree $n$ and
$\mathbb{Z}_{K}$ its ring of integers. It is a classical problem in algebraic
number theory to decide if $K$ is \textit{monogenic field}, or, equivalently,
if $K$ is a field for which there exist an element $\alpha\in\mathbb{Z}_{K}$
such that ring of integers $\mathbb{Z}_{K}$ is of the form $\mathbb{Z}%
_{K}=\mathbb{Z}\left[  \alpha\right]  $. The powers of such element $\alpha$
constitute a \textit{power integral basis}, ie. an integral basis of the form
$\left\{  1,\alpha,\alpha^{2},...,\alpha^{n-1}\right\}  .$ In general, if
$\left\{  1,\omega_{2},...,\omega_{n}\right\}  $ is an integral basis of $K$
and the primitive integer $\alpha\in\mathbb{Z}_{K}$ ($K=\mathbb{Q} (\alpha)$)
is represented in that integral basis as $\alpha=x_{1}+x_{2}\omega
_{2}+...+x_{n}\omega_{n}$, then%
\[
I\left(  \alpha\right)  =\left[  \mathbb{Z}_{K}^{+}:\mathbb{Z}\left[
\alpha\right]  ^{+}\right]  =\left\vert I\left(  x_{2},...,x_{n}\right)
\right\vert
\]
where $\mathbb{Z}_{K}^{+}\ $and $\mathbb{Z}\left[  \alpha\right]  ^{+}$
respectively denote the additive groups of the ring $\mathbb{Z}_{K}$ and the
polynomial ring $\mathbb{Z}\left[  \alpha\right] $. The polynomial $I\left(
X_{2},...,X_{n}\right)  $ is a homogenous polynomial in $n-1$ variables of
degree $\frac{n\left(  n-1\right) }{2}$ with rational integer coefficients
which is called the \textit{index form }corresponding to the integral basis
$\left\{  1,\omega_{2},...,\omega_{n}\right\}  $. The positive rational
integer $I\left(  \alpha\right)  $ is called an \textit{index} of the element
$\alpha$ and it does not depend on $x_{1}.$ Therefore, the primitive integer
$\alpha$ generates a power integral basis if and only if $I\left(
\alpha\right)  =1.$ Consequently, the number field $K$ is monogenic if and
only if the \textit{index form equation}
\begin{equation}
I\left(  x_{2},...,x_{n}\right)  =\pm1 \label{if1}%
\end{equation}
is solvable in rational integers. The problem of determining all generators of
the power integral basis reduces to the resolution of diophantine eq.
(\ref{if1}).

The index form equations are mostly a very complicated diophantine equations
for number fields of large degree $n$ because of high degree $\frac{n\left(
n-1\right) }{ 2}$ of index form and its number of variables $n-1$. In some
particular fields, by studying the structure of index form, it has been found
a correspondence between the index form equation and simpler types of
equations (for a survey see \cite{Gaal}). For example, in \cite{GPP1, GPP2},
I. Ga\'{a}l, A. Peth\H{o} and M. Pohst showed that a resolution of index form
equations in any quartic field can be reduced to the resolution of cubic and
several corresponding Thue equations. In \cite{GP1}, I. Ga\'{a}l, and M. Pohst
extended some basic ideas and developed a method of determined generators of a
power integral basis to relative quartic extension fields $K$ over base fields
$M$. The method is much more complicated than in the absolute case. For
example, instead of Thue equations we obtain relative Thue equations over a
subfield $M$. The generalization of known methods to relative extensions leads
to various nontrivial problems. Those problems occur primarily because a
relative integral basis does not have to exist and also, the ring of integers
of a base field $M$ is not necessarily a unique factorization domain.

Algorithms for solving index form equations have been applied in several
infinite parametric families of certain fields. In particular, I. Ga\'{a}l and
T. Szab\'{o} in \cite{GS1} considered three infinite parametric families of
octic fields that are quartic extensions of imaginary quadratic fields. By
applying the method described in \cite{GP1} and by using some results on
infinite parametric families on relative Thue equations given in \cite{ZIE}
and \cite{JZ}, they found all non-equivalent generators of relative power
integral basis for infinite values of parameter.

In this paper, we consider the following problem. Let $M$ be an imaginary
quadratic field with the ring of integers $\mathbb{Z}_{M}$ and let $\xi$ be a
root of polynomial
\[
f\left(  x\right)  =x^{4}-2cx^{3}+2x^{2}+2cx+1,
\]
where $c\in\mathbb{Z}_{M},$ $c\notin\left\{  0,\pm2\right\}  .$ We consider
infinite family of octic fields $K_{c}=M\left(  \xi\right)  \ $with ring of
integers $\mathbb{Z}_{K_{c}}.$ Since integral basis of $K_{c}$ is not known in
a parametric form, our goal is to determine all generators of relative power
integral basis of $\mathcal{O=}\mathbb{Z}_{M}\left[  \xi\right]  $ over
$\mathbb{Z}_{M}\ $(instead of $\mathbb{Z}_{K_{c}}$ over $\mathbb{Z}_{M}$)$.$

The paper is organized as follows. In Sections \ref{s.PRe} and \ref{sec:3}, we
briefly describe the method of I. Ga\'{a}l, and M. Pohst given in \cite{GP1}
and apply that method to the problem described above. In Section \ref{dio.311}
we show that our problem reduces to solving the system of relative Pellian
equations over $M$ and we prove some results about that system. In Section
\ref{sec:3 CMiBT}, by combining \emph{ congruence method} with an extension of
Bennett's theorem given in \cite{JZ}, we solve the system completely and find
all non-equivalent generators of power integral basis of $\mathcal{O}$ over
$\mathbb{Z}_{M}$ if absolute value of parameter $c$ is large enough
($\left\vert c\right\vert \geq159108$). In the Section \ref{sec:BT}, for
$\left\vert c\right\vert <159108$ we use a theorem of Baker and W\"{u}stholz
and a version of the reduction procedure due to Baker and Davenport. Without
proving that the corresponding linear form $\Lambda\not =0$, we cannot apply
Baker's theory. The proof is rather complicated and involves several cases. We
were not able to perform reduction procedure for all values of $\left\vert
c\right\vert <159108$ because we estimated that it would take $\sim10^{13}$
sec. (in Mathematica). So, we have performed reduction procedure for
$\left\vert c\right\vert \leq200$. In Section \ref{s.Sc} we observe
exceptional cases $c\in S_{{\small c}}.$ For $c\in S_{{\small c}}$ at least
one of the equation of our system of Pellian equations has additional classes
of solutions or there exists only finitely many solutions of those equations.
In Section \ref{s.Sc} we examine whether the order $\mathcal{O=}\mathbb{Z}%
_{M}\left[  \xi\right]  $ admits an absolute power integral basis.

Our main result is the following theorem.

\begin{theorem}
\label{tm:main} Assume that $D$ is a square free positive integer,
$M=\mathbb{Q}\left(  \sqrt{-D}\right)  $ is an imaginary quadratic with ring
of integers $\mathbb{Z}_{M}$, $\xi$ is a root of the polynomial
\[
f\left(  t\right)  =t^{4}-2ct^{3}+2t^{2}+2ct+1,\text{ }%
\]
where $c\in\mathbb{Z}_{M},$ $c\notin\left\{  0,\pm2\right\}  $ and
$K_{c}=M\left(  \xi\right)  $ is an octic field with ring of integers
$\mathbb{Z}_{K_{c}}$. Then all non-equivalent generators of power integral
basis of $\mathcal{O=}\mathbb{Z}_{M}\left[  \xi\right]  $ over $\mathbb{Z}%
_{M}$ are given by
\begin{equation}
\alpha=\xi,\text{ \ }2\xi-2c\xi^{2}+\xi^{3} \label{gu}%
\end{equation}
in each of the following cases:

\begin{description}
\item[i)] for all $D$ and $\left\vert c\right\vert \geq159108;$

\item[ii)] for all $D$, $c\notin S_{c}\ $and $\left\vert c\right\vert
\leq200\ $or $\operatorname{Re}(c)=0;$

\item[iii)] $c=\pm1\ $and $D=1,3,$
\end{description}

where {\small
\begin{equation}
\label{SC1}S_{c}=\{\pm1,\pm\sqrt{-1},\pm1\pm\sqrt{-1},\pm2\pm\sqrt{-1},
\pm1\pm\sqrt{-2},\pm1\pm\sqrt{-3},\frac{\pm1\pm\sqrt{-3}}{2} ,\frac{\pm
3\pm\sqrt{-3}}{2}\},
\end{equation}
} with mixed signs.
\end{theorem}

\begin{proof}
[Proof of Theorem \ref{tm:main}]Immediately from propositions \ref{PropCjv},
\ref{PropL0} ii), \ref{PropRP} and \ref{PropC1}.
\end{proof}

\begin{conjecture}
All non-equivalent generators of power integral basis of $\mathcal{O=}%
\mathbb{Z}_{M}\left[  \xi\right]  $ over $\mathbb{Z}_{M}$ are given by
(\ref{gu}) for all $D$ and $c\notin S_{c}$.
\end{conjecture}

Also, we prove the following theorem.

\begin{theorem}
\label{tm:IG}If $-D\equiv2,3\;(\operatorname{mod}4)$ and all non-equivalent
generators of power integral basis of $\mathcal{O=}\mathbb{Z}_{M}\left[
\xi\right]  $ over $\mathbb{Z}_{M}$ are given by (\ref{gu}), then
$\mathcal{O}$ admits no power integral basis. In particular, in the cases
given in Theorem \ref{tm:main} $\mathcal{O}$ admits no power integral basis.
\end{theorem}

\section{Preliminaries}

\label{s.PRe}

Since we are going to apply the method of I. Ga\'{a}l and M. Pohst given in
\cite{GP1}, we begin with a brief description of it.

Let $M$ be a field of degree $m$ and $K$ its quartic extension field generated
by an algebraic integer $\xi$ over $M$, ie. $K=M\left(  \xi\right)  $.
$\mathbb{Z}_{K}$ and $\mathbb{Z}_{M}$ denotes the ring of integers of $K$ and
$M$, respectively. Assume that a relative minimal polynomial of $\xi$ is given
by
\[
f\left(  t\right)  =t^{4}+a_{1}t^{3}+a_{2}t^{2}+a_{3}t+a_{4}\in\mathbb{Z}%
_{M}\left[  t\right]  .
\]
Also, assume that $d$ is the smallest natural number with the property
$d\mathbb{Z}_{K}\subseteq\mathbb{Z}_{M}\left[  \xi\right]  $ and
$i_{0}=\left[  \mathbb{Z}_{K}^{+}:\mathbb{Z}_{M}\left[  \xi\right]
^{+}\right]  $. Then each $\alpha\in\mathbb{Z}_{K}$ can be represented in the
form
\begin{equation}
\alpha=\frac{1}{d}\left(  a+x\xi+y\xi^{2}+z\xi^{3}\right)  ,\text{
\ }a,x,y,z\in\mathbb{Z}_{M}.\label{ad}%
\end{equation}
The (absolute) index of $\alpha$ can be factorized in the form%
\begin{equation}
I(\alpha)=\left[  \mathbb{Z}_{K}^{+}:\mathbb{Z}_{M}\left[  \alpha\right]
^{+}\right]  \left[  \mathbb{Z}_{M}\left[  \alpha\right]  ^{+}:\mathbb{Z}%
\left[  \alpha\right]  ^{+}\right]  .\label{abin}%
\end{equation}
If the relative index $I_{K/M}(\alpha)=\left[  \mathbb{Z}_{K}^{+}%
:\mathbb{Z}_{M}\left[  \alpha\right]  ^{+}\right]  $ is equal to $1$, then
$\alpha$ can only generate a power integral basis in $K$ (equivalently,
$I(\alpha)=1$).  In \cite{GP1} the following assertion was proved. If
$\alpha\in\mathbb{Z}_{K}$ given by (\ref{ad}) generates a relative power
integral basis of $\mathbb{Z}_{K}$ over $\mathbb{Z}_{M},$ then
\begin{equation}
N_{M/\mathbb{Q}}\left(  F\left(  u,v\right)  \right)  =\pm\frac{d^{6m}}{i_{0}%
},\label{norm}%
\end{equation}
where
\begin{equation}
F\left(  u,v\right)  =u^{3}-a_{2}u^{2}v+\left(  a_{1}a_{3}-4a_{4}\right)
uv^{2}+\left(  4a_{2}a_{4}-a_{3}^{2}-a_{1}^{2}a_{4}\right)  v^{3},\label{cf}%
\end{equation}
is a binary cubic form over $\mathbb{Z}_{M}$,
\begin{equation}
u=Q_{1}\left(  x,y,z\right)  ,\ \ v=Q_{2}\left(  x,y,z\right)  ,\label{UV}%
\end{equation}
and
\begin{align}
Q_{1}\left(  x,y,z\right)   &  =x^{2}-xya_{1}+y^{2}a_{2}+xz\left(  a_{1}%
^{2}-2a_{2}\right)  \label{Q1}\\[0.07in]
&  +yz\left(  a_{3}-a_{1}a_{2}\right)  +z^{2}\left(  -a_{1}a_{3}+a_{2}%
^{2}+a_{4}\right)  ,\nonumber\\
Q_{2}\left(  x,y,z\right)   &  =y^{2}-xz-yza_{1}+a_{2}z^{2},\label{Q2}%
\end{align}
are ternary quadratic forms over $\mathbb{Z}_{M}$.

Note that the equation (\ref{norm}) implies
\begin{equation}
F\left(  u,v\right)  =\delta\varepsilon,\label{fuv}%
\end{equation}
where $\delta$ is an integer in $M$ of the norm $\pm d^{6m}/i_{0}$ and
$\varepsilon$ is an unit in $M$. Hence, the full set of nonassociated elements
of this norm have to be considered.

In order to find all non-equivalent generators of power integral basis of
$\mathcal{O}$, the first step consists of solving the equation (\ref{fuv}),
ie. determining all (nonassociated) pairs $(u,v)\in\mathbb{Z}_{M}^{2}$ such
that all solutions of (\ref{fuv}) are of the form $(\eta u,\eta v)$, where
$\eta\in M$ is an unit. In the next step, we have to find all $\left(
x,y,z\right)  \in\mathbb{Z}_{M}^{3}$ corresponding to a fixed solution $(u,v)$
by solving the system (\ref{UV}). So, for a given solution $(u,v)$ of
(\ref{fuv}), we solve the following equation
\begin{equation}
Q_{0}\left(  x,y,z\right)  =uQ_{2}\left(  x,y,z\right)  -vQ_{1}\left(
x,y,z\right)  =0.\label{Q0}%
\end{equation}
Using the arguments of Siegel \cite[p.264]{SIG} (see also \cite{POH}), it is
possible to decide if (\ref{Q0}) has nontrivial solutions and if so, all
solutions of (\ref{Q0}) can be given in a parametric form (with two parameters
$p$ and $q$). By substituting these parametric representations of $u$ and $v$
into the original system (\ref{UV}), it can be shown that at least one of the
equations in (\ref{UV}) is a quartic Thue equation over $\mathbb{Z}_{M}$. By
solving that Thue equation, we are able to determine all parameters
$(p,q)\in\mathbb{Z}_{M}^{2}$ up to unit factors in $M$. Hence, we can
calculate all $\left(  x,y,z\right)  \in\mathbb{Z}_{M}^{3}$ up to a unit
factor of $M$, as well. Then all generators of power integral basis of
$\mathbb{Z}_{K}$ over $\mathbb{Z}_{M}$ are of the form
\[
\alpha=\frac{1}{d}\left(  a+\eta\left(  x\xi+x\xi^{2}+x\xi^{3}\right)
\right)  ,
\]
where $a\in\mathbb{Z}_{M}$, and the unit $\eta\in M$ is arbitrary.
Consequently, all non-equivalent generators of power integral basis of
$\mathcal{O}$ over $\mathbb{Z}_{M}$ are given by $\alpha=\frac{1}{d}\left(
x\xi+x\xi^{2}+x\xi^{3}\right)  $. For more details see \cite{GP1}.

Our purpose is to describe the relative power integral bases of either
$\mathcal{O=}\mathbb{Z}_{K}$ over $\mathbb{Z}_{M}$ (if the integer basis of
$K$ is known) or of $\mathcal{O=}\mathbb{Z}_{M}\left[  \xi\right]  $ over
$\mathbb{Z}_{M}$ (otherwise). Note that in the later case $\xi$ itself is a
generator of a relative power integral basis but we wonder if there exist any
other generators of power integral bases. Note that in case $\mathcal{O=}%
\mathbb{Z}_{M}\left[  \xi\right]  ,$ we have $i_{0}=d=1.\ $Consequently,
equation (\ref{norm}) is of the form $F\left(  u,v\right)  =\varepsilon,$
where $\varepsilon$ is unit in $M$ and non-equivalent generators of power
integral basis of $\mathcal{O}$ over $\mathbb{Z}_{M}$ are of the form
$\alpha=x\xi+y\xi^{2}+z\xi^{3},$ $\left(  x,y,z\right)  \in\mathbb{Z}_{M}^{3}$.

\section{Resolution of relative cubic equation}\label{sec:3}
Let $D$ be a square free positive integer and let $M=\mathbb{Q}%
\left(  \sqrt{-D}\right)  $ be an  imaginary quadratic with ring of integers
$\mathbb{Z}_{M}$. Let $\xi$ be a root of polynomial%
\begin{equation}
f\left(  t\right)  =t^{4}-2ct^{3}+2t^{2}+2ct+1,\text{ }\label{mp}%
\end{equation}
where $c\in\mathbb{Z}_{M},$ $c\notin\left\{  0,\pm2\right\}$.  We consider an 
infinite family of octic fields $K_{c}=M\left(  \xi\right)  \ $with ring of
integers $\mathbb{Z}_{K_{c}}$. It is easy to see that if $c=0,\pm2$, then 
$f\left(  t\right) $ is a reducible polynomial and so  $K_{c}$
is not an octic field. Therefore from now on we  assume that $c\in\mathbb{Z}
_{M}\backslash\left\{  0,\pm2\right\}$. Since the integral basis of $K_{c}$ is not known in a parametric
form, our goal is to determine all generators $\alpha$ of relative power
integral basis of $\mathcal{O=}\mathbb{Z}_{M}\left[  \xi\right]  $ over
$\mathbb{Z}_{M}\ $(instead of $\mathbb{Z}_{K_{c}}$ over $\mathbb{Z}_{M}$)$.$
In this case the equation (\ref{fuv}) is of the form
\begin{equation}
F\left( u,v\right) =\left( u+2v\right) \left( u-2\left( c+1\right) v\right)
\left( u+2\left( c-1\right) v\right) =\varepsilon ,  \label{Ff}
\end{equation}%
where $\varepsilon$ is an unit in $M,$ ie. $\varepsilon \in \{\pm 1,\pm i,\pm
\omega ,\pm \omega ^{2}\}\cap \mathbb{Z}_{M}$ and (\ref{Q1}),
 (\ref{Q2}) can be rewritten as  
\begin{align*}
Q_{1}\left( x,y,z\right) & =x^{2}+2cxy+2y^{2}+4\left( c^{2}-1\right)
xz+6cyz+z^{2}\left( 4c^{2}+5\right) \\
Q_{2}\left( x,y,z\right) & =y^{2}-xz+2cyz+2z^{2}.
\end{align*}%
According to (\ref{Ff}) we conclude that $u-2v$, $u-2\left( c+1\right) v$, $u+2\left( c-1\right) v$ are units in $\mathbb{Z}_{M}$ and  that implies $v=0$. Therefore, all solutions
of  (\ref{Ff}) are given by $\left( u,v\right) =\left( \eta
,0\right) $, where $\eta $ is an unit in $\mathbb{Z}_{M}$.

\section{Simultaneous Pellian equations}\label{dio.311}

In this part we show that solving the equation (\ref{Q0}) for $\left(  u,v\right)
=\left(  \eta,0\right) $  can be reduced to solving a system of
simultaneous Pellian equations. 

Let $c\in\mathbb{Z}_{M}\backslash\{0,\pm2\}$.
Since $v=0$, the equation (\ref{Q0}) implies
\begin{equation}
Q_{2}(x,y,z)=y^{2}-xz+2cyz+2z^{2}=0,\label{e1}%
\end{equation}
and $(x_{0},y_{0},z_{0})=(2,0,1)$ is one nontrivial solution of (\ref{e1}).
Therefore, all solutions can be parameterized by%
\begin{equation}
x=2r+p,\ y=q,\ z=r,\label{e2}%
\end{equation}
where $p,q,r\in M$ and $r\not =0$. By substituting (\ref{e2}) into (\ref{e1}),
we obtain%
\begin{equation}
q^{2}=r(p-2cq).\label{e3}%
\end{equation}
Further, if we multiply (\ref{e2}) by $k=p-2cq$, we get%
\begin{equation}
kx=2q^{2}+p^{2}-2cqp,\ ky=qp-2cq^{2},\ kz=q^{2}.\label{e4}%
\end{equation}
We can assume that $k,p,q\in\mathbb{Z}_{M}$ and since the corresponding
determinant equals $1$, the parameter $k$ must be an unit in $\mathbb{Z}_{M}$.
Now, by substituting $kx,ky,kz$ given by (\ref{e4}) into the equation
$Q_{1}(x,y,z)=\eta$ ($\eta$ is an unit in $\mathbb{Z}_{M}$) we obtain%
\begin{equation}
p^{4}-2cp^{3}q+2p^{2}q^{2}+2cpq^{3}+q^{4}=\mu,\label{e5}%
\end{equation}
where $\mu=k^{2}\eta$ is an unit in $\mathbb{Z}_{M}.$ This is a
relative Thue equation over $\mathbb{Z}_{M}$ and it can be transformed into a system of Pellian equations
\begin{align}
cV^{2}-(c+2)U^{2} &  =-2\mu,\label{e6}\\
\ (c-2)U^{2}-cZ^{2} &  =-2\mu,\label{e7}%
\end{align}
by putting%
\begin{equation}\label{UVZ}
U   =p^{2}+q^{2},\ 
V   =p^{2}+2pq-q^{2},\ 
Z   =-p^{2}+2pq+q^{2}. \end{equation}
Both of equations (\ref{e6}) and (\ref{e7}) are of the same form as the
equation already studied in \cite{JZ}, ie. of the form%
\begin{equation}\label{en1}
(k-1)x^{2}-(k+1)y^{2}=-2\mu.
\end{equation}

\begin{proposition}
[{\cite[Proposition 5.2]{JZ}}]\label{l.1}Let $k\in\mathbb{Z}_{M}$ and let 
$\mu\in\mathbb{Z}_{M}$ be an unit. Suppose $|k|\geq2$ or $k$ is not an element of
the set
\[
S=\{0,\pm1,\pm\sqrt{-1},\pm1\pm\sqrt{-1},\pm\sqrt{-2},\pm\sqrt{-3},\pm
\omega,\pm\omega^{2}\},
\]
with mixed signs, where $\omega=\frac{-1+\sqrt{-3}}{2}.$ If the  equation
(\ref{en1}) is solvable, then
\[
\mu\in\{1,-1,\omega,\omega^{2}\}.
\]
All solutions are of the form $(x,y)=(\pm x_{m},\pm y_{m})$,
with mixed signs, where the sequences $(x_{m})$ and $(y_{m})$ are given with
the recurrence relations%
\begin{align}
&  x_{0}=\epsilon,x_{1}=\epsilon(2k+1),\ x_{m+2}=2kx_{m+1}-x_{m}%
,\ m\geq0,\label{e10}\\
&  y_{0}=\epsilon,y_{1}=\epsilon(2k-1),\ y_{m+2}=2ky_{m+1}-y_{m}%
,\ m\geq0,\label{e11}%
\end{align}
where $\epsilon=1,\sqrt{-1},\omega^{2},\omega$ corresponds to $\mu
=1,-1,\omega,\omega^{2}$, respectively.
\end{proposition}

For $k=c+1$ Proposition \ref{l.1} implies that if
{\small
$$
c\not \in \{{-1,-1\pm\sqrt{{ -1}},\pm\sqrt{{ -1}%
},-2\pm}\sqrt{{ -1}}{ ,-1\pm}\sqrt{{ -2}}{ ,-1\pm
}\sqrt{{-3}}{ ,}\frac{{ -1\pm}\sqrt{{-3}}%
}{{2}}{ ,}\frac{{ -3\pm}\sqrt{{ -3}}}{{2}}\},
$$
}
then all solutions $(V,U)$ of (\ref{e6}) are of the form $(\pm v_{m},\pm
u_{m})$ where%
\begin{align}
&  v_{0}=\epsilon,v_{1}=\epsilon(2c+3),\ v_{m+2}=(2c+2)v_{m+1}-v_{m}%
,\ m\geq0,\nonumber\\
&  u_{0}=\epsilon,u_{1}=\epsilon(2c+1),\ u_{m+2}=(2c+2)u_{m+1}-u_{m}%
,\ m\geq0.\label{e12}%
\end{align}
Similarly, if $k=c-1$ and%
\[
c\not \in \{{\small 1,1\pm}\sqrt{{\small -1}}{\small ,\pm}\sqrt
{{\small -1}}{\small ,2\pm}\sqrt{{\small -1}}{\small ,1\pm}\sqrt{{\small -2}%
}{\small ,1\pm}\sqrt{-3}{\small ,}\frac{{\small 1\pm}\sqrt{{\small -3}}%
}{{\small 2}}{\small ,}\frac{{\small 3\pm}\sqrt{{\small -3}}}{{\small 2}}\},
\]
then all solutions $(U,Z)$ of (\ref{e7}) are of the form $(\pm u_{n}^{\prime
},\pm z_{n})$ where
\begin{align}
&  u_{0}^{\prime}=\epsilon,u_{1}^{\prime}=\epsilon(2c-1),\ u_{n+2}^{\prime
}=(2c-2)u_{n+1}^{\prime}-u_{n}^{\prime},\ n\geq0,\label{e14}\\
&  z_{0}=\epsilon,z_{1}=\epsilon(2c-3),\ z_{n+2}=(2c-2)z_{n+1}-u_{n}
,\ n\geq0.\nonumber
\end{align}
Finally,  if $c\not \in S_{c}$ where $S_{c}$ is given
in (\ref{SC1}) and the system of equations (\ref{e6}) and (\ref{e7}) is
solvable, then  Proposition \ref{l.1} implies
$$\mu\in\{1,-1,\omega,\omega^{2}\}.$$
Furthermore, if $\left(  U,V,Z\right)  $ is a solution of that system, then
\[
U=\pm u_{m}=\pm u_{n}^{\prime}
\]
for some $n,m\in\mathbb{N}_{0}$, with mixed signs,  where $u_m$, $u_n'$ are given by (\ref{e12}),  (\ref{e14}) and   $\epsilon
=1,\sqrt{-1},\omega^{2},\omega$ corresponds to $\mu=1,-1,\omega,\omega^{2}$.
 Evidently, $U=\pm u_{0}=\pm u_{0}^{\prime}%
=\pm\epsilon$. So, the next step consists of determining eventual
intersections of sequences $(\pm u_{m})$ and $(\pm u_{n}^{\prime})$ for
$m,n\geq1$.

\section{Proof of the main Theorem for $|c|\geq159\,108$}\label{sec:3 CMiBT}

In this section we apply the \emph{congruence method} introduced in \cite{DP} to obtain lower bound for $|U|$.
Combining that result with a generalization of \emph{Bennett's theorem}, we are able to solve the system
(\ref{e6}) and (\ref{e7}) for large values of $|c|$. 

\subsection{A lower bound for a solution}

\begin{definition}
Let $a,b,d\in\mathbb{Z}_{M}$ and $d\not =0$. We say that $a$ is
\emph{congruent} $b$ \emph{modulo} $d$ if there exists $x\in\mathbb{Z}_{M}$
such that $a-b=dx$. We write $a\equiv b\left(  \operatorname{mod}d\right)  $.
\end{definition}

\begin{lemma}
\label{l.13}Let $|c|\geq2$. Sequences $(u_{m})$ and $(u_{n}^{\prime})$ given
by (\ref{e12}) and (\ref{e14}) satisfy the following inequalities%
\begin{equation}
(2|c|-3)^{m}\leq|u_{m}|\leq(2|c|+3)^{m},\ (2|c|-3)^{n}\leq|u_{n}^{\prime}%
|\leq(2|c|+3)^{n}, \label{e15}%
\end{equation}
for $m,n\geq0$.
\end{lemma}

\begin{proof}
The inequality for $|u_{n}^{\prime}|$ is given in \cite[ Lemma 5.5.]{JZ}.
Similarly, we prove the other one. First, we show by induction that
$(|u_{m}|)$ is a growing sequence. Evidently,%
\[
|u_{1}|=|2c+1|\geq2|c|-1\geq1=|u_{0}|.
\]
If $|u_{m}|\geq|u_{m-1}|$ for some $m\in\mathbb{N}$, then%
\[
|u_{m+1}|\geq|2c+2||u_{m}|-|u_{m-1}|\geq|2c+2||u_{m}|-|u_{m}|=(2|c|-3)|u_{m}%
|\geq|u_{m}|.
\]
Since, $|u_{0}|=(2|c|-3)^{0}$ and $|u_{1}|\geq2|c|-1\geq(2|c|-3)^{1}$, the
previous inequality $|u_{m+1}|\geq(2|c|-3)|u_{m}|$ also implies that
$|u_{m}|\geq(2|c|-3)^{m}$ for $m\geq0$. Also, since%
\[
|u_{0}|=(2|c|+3)^{0},\ |u_{1}|=|2c+1|\leq(2|c|+1)\leq(2|c|+3)^{1}.
\]
and%
\[
|u_{m+1}|\leq(2|c|+2)|u_{m}|+|u_{m-1}|\leq(2|c|+3)|u_{m}|,
\]
we obtain that $|u_{m}|\leq(2|c|+3)^{m}$ for $m\geq0$.
\end{proof}

\begin{lemma}
\label{l.14}Sequences $(\pm u_{m})$ and $(\pm u_{n}^{\prime})$ given by
(\ref{e12}) and (\ref{e14}) satisfy the following congruences%
\begin{align}
u_{m}  &  \equiv\epsilon(1+m(m+1)c)\pmod{4c^2},\label{e18}\\
u_{n}^{\prime}  &  \equiv(-1)^{n}\epsilon(1-n(n+1)c)\pmod{4c^2},
\end{align}
for $m,n\geq0$.
\end{lemma}

\begin{proof}
The congruence relation for $u_{n}^{\prime}$ has already been proved in
\cite[Lemma 6.2.]{JZ}. The other relation can be easily shown by induction. Recall that 
 $u_{0}=\epsilon$ and $u_{1}=\epsilon(2c+1)$. Hence, (\ref{e18}) is true for $m=0,1$. Now, assume that
$u_{k}\equiv\epsilon(1+k(k+1)c)\pmod{4c^2}$, for $k<m$ and $m\geq2$. We
obtain%
\begin{align*}
u_{m}  &  =(2c+2)u_{m-1}-u_{m-2}\equiv(2c+2)\epsilon(1+(m-1)mc)-\epsilon
(1+(m-2)(m-1)c)\\
&  \equiv\epsilon(1+2m(m-1)c^{2}+m(m+1)c)\equiv\epsilon(1+m(m+1)c)\pmod{4c^2},
\end{align*}
for $m\geq2$.
\end{proof}

\begin{proposition}\label{p.15}
Let  $c\notin S_{c}$. If $u_{m}=\pm u_{n}^{\prime}$, then%
$$m\geq\sqrt{2|c|+0.25}-0.5\ \text{ or }\ n\geq\sqrt{2|c|+0.25}-0.5\ \text{ or
}\ m=n=0.$$

\end{proposition}

\begin{proof}
If $u_{m}=\pm u_{n}^{\prime}$, then Lemma \ref{l.14} implies that%
\[
\epsilon(1+m(m+1)c)\equiv\pm(-1)^{n}\epsilon(1-n(n+1)c)\pmod{4c^2}.
\]
Therefore we have the following congruence relation
\[
\epsilon(1\mp(-1)^{n})\equiv0\pmod{2c}.
\]
If $\epsilon(1\mp(-1)^{n})\not =0$, then $|\epsilon(1\mp(-1)^{n})|=2$ and
$|c|=1$, which is not possible. So, we conclude that $\mp(-1)^{n}=-1$ and%
\[
\epsilon(1+m(m+1)c)\equiv\epsilon(1-n(n+1)c)\pmod{4c^2}.
\]
Furthermore,%
\begin{equation}
\epsilon\left(  \frac{m(m+1)}{2}+\frac{n(n+1)}{2}\right)  \equiv
0\pmod{2c}.\label{A}%
\end{equation}
Consider the algebraic integer
\[
A=\epsilon\left(  \frac{m(m+1)}{2}+\frac{n(n+1)}{2}\right)  .
\]
It is clear that $A\not =0$ for $m>0$ or $n>0$. So, (\ref{A}) implies that
$|A|\geq2|c|$. Hence,
\[
m(m+1)\geq2|c|\ \text{ or }\ n(n+1)\geq2|c|,
\]
ie.%
$$
m\geq\sqrt{2|c|+0.25}-0.5\ \text{ or }\ n\geq\sqrt{2|c|+0.25}-0.5.$$
\end{proof}

Finally, the previous proposition yields a lower bound for a nontrivial
solution of equations (\ref{e6}) and (\ref{e7}).

\begin{corollary}
\label{c.16}Let $c\notin S_{c}$. If $U\in\mathbb{Z}_{M}\backslash\{\pm
\epsilon\}$ is a solution of the system of equations (\ref{e6}) and
(\ref{e7}), then%
\[
|U|\geq(2|c|-3)^{\sqrt{2|c|+0.25}-0.5}.
\]

\end{corollary}

\begin{proof}
If follows straight away from Lemma \ref{l.13} and Proposition \ref{p.15}.
\end{proof}

\subsection{An upper bound for a solution}

The number of solutions of simultaneous Pellian equations can be bounded by a
theorem on simultaneous approximations by rationals to the square roots of
rationals near $1$ introduced by M. Bennett in \cite{Ben}. In fact, we need its
generalization for imaginary quadratic fields stated and proved in \cite{JZ}.
Namely, we use the following theorem:
\begin{theorem}
[{\cite[Theorem 7.1]{JZ}}]\label{t.borka} Let $\theta_{i}=\sqrt{1+\frac{a_{i}%
}{T}}$ for $1\leq i\leq m$, with $a_{i}$ pairwise distinct imaginary quadratic
integers in $K:=\mathbb{Q}(\sqrt{-D})$ with $0<D\in\mathbb{Z}$ for
$i=0,\ldots,m$ and let $T$ be an algebraic integer of $K$. Furthermore, let
$M:=\max|a_{i}|$, $|T|>M$ and $a_{0}=0$ and%
\[%
\begin{array}
[c]{cc}%
l=c_{m}\frac{\left(  m+1\right)  ^{m+1}}{m^{m}}\cdot\frac{|T|}{|T|-M}, &
\text{ \ \ \ \ }L=|T|^{m}\frac{\left(  m+1\right)  ^{m+1}}{4m^{m}\prod_{0\leq
i<j\leq m}|a_{j}-a_{i}|^{2}}\cdot\left(  \frac{|T|-M}{|T|}\right)  ^{m},
\end{array}
\]%
\[%
\begin{array}
[c]{cc}%
p=\sqrt{\frac{2|T|+3M}{2|T|-2M}}, & \text{ \ \ \ \ \ \ \ }P=|T|\cdot
2^{m+3}\frac{\prod_{0\leq i<j\leq m}|a_{i}-a_{j}|^{2}}{\min_{i\not =j}%
|a_{i}-a_{j}|^{m+1}}\cdot\frac{2|T|+3M}{2|T|},
\end{array}
\]
where $c_{m}=\frac{3\Gamma\left(  m-\frac{1}{2}\right)  }{4\sqrt{\pi}%
\Gamma\left(  m+1\right)  },$ such that $L>1$, then
\[
\max\left(  \left\vert \theta_{1}-\frac{p_{1}}{q}\right\vert ,\ldots
,\left\vert \theta_{m}-\frac{p_{m}}{q}\right\vert \right)  >cq^{-\lambda}%
\]
for all algebraic integers $p_{1},\ldots,p_{m},q\in K$, where
\begin{align*}
\lambda &  =1+\frac{\log P}{\log L}\quad\text{and}\\
C^{-1}  &  =2mpP\left(  \max\left\{  1,2l\right\}  \right)  ^{\lambda-1}.
\end{align*}

\end{theorem}

The first step in the application of Theorem \ref{t.borka} consists of
choosing suitable values for $\theta_{1}$ and $\theta_{2}$. Let $(U,V,Z)\in
\mathbb{Z}_{M}^{3}$ be a solution of system of Pellian equations (\ref{e6})
and (\ref{e7}). The
candidates for $\theta_{1}$ and $\theta_{2}$ are%
\begin{equation}
\theta_{1}^{(1)}=\pm\sqrt{\frac{c+2}{c}},\ \theta_{2}^{(1)}=\pm\sqrt
{\frac{c-2}{c}},\ \theta_{1}^{(2)}=-\theta_{1}^{(1)},\ \theta_{2}%
^{(2)}=-\theta_{2}^{(1)},\label{e.theta}%
\end{equation}
where the signs are chosen such that%
\[
|V-\theta_{1}^{(1)}U|<|V-\theta_{1}^{(2)}U|\ \ \text{and}\ \ \ |Z-\theta
_{2}^{(1)}U|<|V-\theta_{2}^{(2)}U|.
\]
The next lemma shows that $\frac{V}{U}$ and $\frac{Z}{U}$ are good
approximations to the algebraic numbers $\theta_{1}^{(1)}$ and $\theta
_{2}^{(1)}$.

\begin{lemma}
\label{l.18}Let $|c|>2$. If $(U,V,Z)\in\mathbb{Z}_{M}^{3}$ is a solution of
(\ref{e6}) and (\ref{e7}), then%
\[
\max\left\{  \left\vert \theta_{1}^{(1)}-\frac{V}{U}\right\vert ,\left\vert
\theta_{2}^{(1)}-\frac{Z}{U}\right\vert \right\}  \leq\frac{2}{\sqrt
{|c|(|c|-2)}}|U|^{-2}.
\]

\end{lemma}

\begin{proof}
We have%
\begin{align*}
|V-\theta_{1}^{(2)}U|  &  \geq\frac{1}{2}(|V-\theta_{1}^{(1)}U|+|V-\theta
_{1}^{(2)}U|)\\
&  \geq\frac{1}{2}|U||\theta_{1}^{(1)}-\theta_{1}^{(2)}|=|U|\left\vert
\sqrt{\frac{c+2}{c}}\right\vert \geq|U|\sqrt{\frac{|c|-2}{|c|}},
\end{align*}
This implies%
\[
\left\vert \theta_{1}^{(1)}-\frac{V}{U}\right\vert =\left\vert \frac{c+2}%
{c}-\frac{V^{2}}{U^{2}}\right\vert \left\vert \theta_{1}^{(2)}-\frac{V}%
{U}\right\vert ^{-1}\leq\frac{2}{|c||U|^{2}}\sqrt{\frac{|c|}{|c|-2}}=\frac
{2}{\sqrt{|c|(|c|-2)}}|U|^{-2}.
\]
Inequality
\[
\left\vert \theta_{2}^{(1)}-\frac{Z}{U}\right\vert \leq\frac{2}{\sqrt
{|c|(|c|-2)}}|U|^{-2}%
\]
is proved in \cite[Lemma 8.1]{JZ}.
\end{proof}

The inputs of Theorem \ref{t.borka} are $m=2$, $\theta_{1}=\theta_{1}^{(1)}$,
$\theta_{2}=\theta_{2}^{(1)}$, $a_{1}=2$, $a_{2}=-2$, $M=2$, $T=c$, with
$|T|=|c|>2$,%
\begin{align*}
l &  =\frac{27}{64}\frac{|c|}{|c|-2},\text{ \ }L=\frac{27}{4096}%
(|c|-2)^{2}>1\ \text{if }\ |c|\geq15,\\
p &  =\sqrt{\frac{|c|+3}{|c|-2}},\text{ \ }P=1024(|c|+3),\\
\lambda &  =1+\frac{\log1024+\log(|c|+3)}{\log27-\log{4096}+2\log(|c|-2)},\\
C^{-1} &  =4pP(\max\{1,\frac{27}{32}\frac{|c|}{|c|-2}\})^{\lambda
-1}=4pP=4096(|c|+3)\sqrt{\frac{|c|+3}{|c|-2}}\ \text{if }\ |c|\geq13.
\end{align*}
Finally, Lemma \ref{l.18} and Theorem \ref{t.borka} for $p_{1}=V$, $p_{2}=Z$
and $q=U$ give us the following inequality%
\[
\frac{2}{|U|^{2}\sqrt{|c|(|c|-2)}}\geq\max\left\{  \left\vert \theta_{1}%
-\frac{V}{U}\right\vert ,\left\vert \theta_{2}-\frac{Z}{U}\right\vert
\right\}  >C|U|^{-\lambda},
\]
ie.
\[\frac{2C^{-1}}{\sqrt{|c|(|c|-2)}}>|U|^{2-\lambda}.
\]
So, if $2-\lambda>0$ the obtained upper bound for  $|U|$ is%
\begin{equation}
\log|U|<\frac{\log(\frac{2C^{-1}}{\sqrt{|c|(|c|-2)}})}{2-\lambda}.\label{e.19}%
\end{equation}
We now examine the condition $f(|c|)=2-\lambda>0$, where%
\[
f(t)=1-\frac{\log1024+\log(t+3)}{\log27-\log{4096}+2\log(t-2)},\ t>2.
\]
For $t>15$, $f(t)$ is a strictly growing function and since $\lim
_{t\rightarrow\infty}f(t)=\frac{1}{2}$, there exists $t_{0}$ such that
$f(t)>0$ for $t\geq t_{0}$. Since, $f(155\,352)>0$ and $f(155\,351)<0$, we
conclude that  the condition $2-\lambda>0$ is
fulfilled for $|c|\geq155\,352$. Now, we use the lower bound for $|U|$ given in Corollary \ref{c.16} and 
obtain%
\begin{equation}
\log|U|\geq(\sqrt{2|c|+0.25}-0.5)\log(2|c|-3),\ |c|>2.\label{e.20}%
\end{equation}
Comparing (\ref{e.19}) and (\ref{e.20}) we get the  inequality%
\[
(\sqrt{2|c|+0.25}-0.5)\log(2|c|-3)<\frac{\log(8192\cdot\frac{(|c|+3)}%
{\sqrt{|c|(|c|-2)}}\sqrt{\frac{|c|+3}{|c|-2}})}{1-\frac{\log1024(|c|+3)}%
{\log\frac{27}{{4096}}(|c|-2)^{2}}},
\]
which does not hold for $|c|\geq159\,108$. Therefore, we have proved the following assertion.

\begin{proposition}
\label{prop2}For $\left\vert c\right\vert \geq159\,108,$ the only solutions of
the system (\ref{e6}) and (\ref{e7}) are $\left(  U,V,Z\right)  =\left(
\pm\varepsilon,\pm\varepsilon,\pm\varepsilon\right)  $ with mix signs and
$\varepsilon=1,$ $i,$ $\omega,$ $\omega^{2}$ corresponding to $\mu
=1,-1,\omega,$ $\omega^{2},$ respectively.
\end{proposition}

Let $\left(  p,q\right)  \in\mathbb{Z}_{M}^{2}$ be a solution of 
(\ref{e5}) and let $\left\vert c\right\vert \geq159\,108$. From Proposition
\ref{prop2} and equations in (\ref{UVZ}), we have%
\[
U=p^{2}+q^{2}=\pm\varepsilon,\text{ \ }V=p^{2}+2pq-q^{2}=\pm\varepsilon,\text{
\ }Z=-p^{2}+2pq+q^{2}=\pm\varepsilon,
\]
where $\varepsilon=1,$ $i,$ $\omega,$ $\omega^{2}.$ Adding $V$ and $Z$ yields
$2pq=0,\pm\varepsilon.$ Since $\left\vert 2pq\right\vert \geq2$ or $\left\vert
2pq\right\vert =0,$ we have $2pq=0.$ Hence, either $p$ or $q$ is equal to $0$
which implies $p^{4}=\mu$ and $q=0$\ or $q^{4}=\mu$ and $p=0$, where $\mu
\in\{1,-1,\omega,\omega^{2}\}$. Therefore the following
theorem follows immediately.

\begin{theorem}
\label{Trm :Thue} Let $c\notin S_{c}$.  If the equation (\ref{e5}) is solvable in 
$(p,q)\in\mathbb{Z}_{M}^{2},$ then $\mu\in\{1,-1,\omega,\omega^{2}\}$ where
$\omega=\frac{1}{2}(-1+\sqrt{-3})$. Furthermore, if $|c|\geq159\,108$, then all
solutions of  (\ref{e5}) are given by
\begin{enumerate}
\item $(p,q)\in\{(0,\pm1),(\pm1,0),(0,\pm i),(\pm i,0)\}\cap\mathbb{Z}_{M}%
^{2}$ if $\mu=1$;

\item $(p,q)\in\{(0,\pm\omega),(\pm\omega,0)\}\cap\mathbb{Z}_{M}^{2}$ if
$\mu=\omega$;

\item $(p,q)\in\{(0,\pm\omega^{2}),(\pm\omega^{2},0)\}\cap\mathbb{Z}_{M}^{2}$
if $\mu=\omega^{2}.$
\end{enumerate}
\end{theorem}
Note, that if $\mu=-1$ and $|c|\geq
159\,108$, then  there is no solution of (\ref{e5}).
Equations in (\ref{e4}) and Theorem \ref{Trm :Thue} imply the following proposition right away. 

\begin{proposition}
\label{PropCjv}If $|c|\geq159\,108,$ then all non-equivalent generators of
power integral basis of $\mathcal{O=}\mathbb{Z}_{M}\left[  \xi\right]  $ over
$\mathbb{Z}_{M}$ are $\alpha=\xi,$\ $2\xi-2c\xi^{2}+\xi^{3}$.
\end{proposition}

\begin{remark}
\label{Rem1}Note that $\left(  U,V,Z\right)  =\left(  \pm\varepsilon
,\pm\varepsilon,\pm\varepsilon\right)  $ with mix signs and $\varepsilon=1,$
$i,$ $\omega,$ $\omega^{2}$ corresponding to $\mu=1,-1,\omega,$ $\omega^{2},$
respectively, are solutions of the system (\ref{e6}) and (\ref{e7}) for all
$c\in\mathbb{Z}_{M}.$ This implies that $\alpha=\xi,$\ $2\xi-2c\xi^{2}+\xi
^{3}$ are non-equivalent generators of power integral basis for all
$c\in\mathbb{Z}_{M}\backslash\left\{  0,\pm2\right\}.$
\end{remark}

\section{Applying Baker's theory for $|c|<159\,108$ }\label{sec:BT}

 The famous theorem of Baker and W\"{u}stholz from \cite{BW} says:
\begin{theorem}\label{l.bw} If $\Lambda = b_1\alpha_1 + \cdots+  b_l\alpha_l\not= 0$, where $\alpha_1 , \ldots,\alpha_l$ are algebraic integers and
$b_1, \ldots , b_l$ are rational integers, then
$$\log |\Lambda| \geq-18(l + 1)!l^{l+1}(32d)^{l+2}h′(\alpha_1)\cdots h′(\alpha_l) \log(2ld)\log B,$$
where $B =\max\{|(b_1|, \ldots , |b_l|\}$, $d$ is the degree of the number field generated by $\alpha_1, \ldots ,\alpha_l$
over the rationals $\Q$,
$$ h'(\alpha) = \max\{ h(\alpha),\frac{1}{d}|\log \alpha|,\frac{1}{d}\},$$
and $h(\alpha)$ denotes the standard logarithmic Weil height.
\end{theorem}

The previous theorem can be applied on finding intersections of binary recursive sequences close to sequences of the form $\alpha\beta^m$, where $\alpha,\beta$ are algebraic integers. So, first we make sure that our sequences $u_m$ and $u_n'$ are close to that form.

Assume that $|c|<159\,108$, $c\notin S_{c}$ and $\on{Re}(c)\geq 0$. Indeed, if $\on{Re}(c)<0$, then by replacing $c$ in the system of equations (\ref{e6}), (\ref{e7}) by $-c$, we obtain the system
$$ (c -2)U^2 - cV^2 = -2\mu,\ cZ^2 - (c + 2) U^2 =-2\mu,$$
which corresponds to the initial system (\ref{e6}), (\ref{e7}) by switching places of $Z$ and $V$. Let us agree that the square root of a complex number $z = re^{i\varphi}$, $-\pi<\varphi\leq\pi$  is given by
$$ \sqrt{z}=\sqrt{r}e^{i\frac{\varphi}{2}},$$
ie. the one with a positive real part (or the principal square root).

In Section \ref{dio.311} we showed that $U=\pm u_m$, where $U$ is the solution of (\ref{e6}) and the sequence $(u_m)$ is given by (\ref{e12}). Solving the recursion in (\ref{e12}) yields an explicit expression for $u_m$:
\begin{eqnarray}
u_m=\epsilon\frac{1}{2\sqrt{c(c+2)}}\left((c+\sqrt{c(c+2)})(c+1+\sqrt{c(c+2)})^m\right. \nonumber \\ \left.-(c-\sqrt{c(c+2)})(c+1-\sqrt{c(c+2)})^m \right).\label{e.98}
\end{eqnarray}
Since $Re(c)\geq 0$, $|c+1+\sqrt{c(c+2)}|\cdot|c+1-\sqrt{c(c+2)}|=1$ and $|c+1+\sqrt{c(c+2)}|\not=|c+1-\sqrt{c(c+2)}|$ for $c\not=0,-1,-2$, we have $|c+1+\sqrt{c(c+2)}|> 1$ (and $|c+1-\sqrt{c(c+2)}|< 1$). So, we put
\begin{equation}\label{e.100p}
P=\frac{1}{\sqrt{c+2}}(c+\sqrt{c(c+2)})(c+1+\sqrt{c(c+2)})^m.
\end{equation}
Through algebraic manipulation it can be shown
\begin{equation}\label{e90} u_m =\frac{\epsilon}{2\sqrt{c}}(P +\frac{2c}{c+1}P^{-1}), \end{equation}
since $\sqrt{c (c + 2)}= \sqrt{c}\sqrt{(c + 2)}$ if $\on{Re}(c)\geq 0$ (because $\sqrt{z_1z_2}=\sqrt{z_1}\sqrt{z_2}$ is not true in general) and
$$ P^{-1}=\frac{\sqrt{c + 2}}{2c}(\sqrt{c(c+2)}-c)(c+1-\sqrt{c(c+2)})^m.$$
Analogously, $U=\pm u'_n$ is the solution of (\ref{e7}) where the sequence $(u'_n)$ is given by (\ref{e14}) and it's explicit expression is
\begin{eqnarray}
u'_n=\epsilon\frac{1}{2\sqrt{c(c-2)}}\left((c+\sqrt{c(c-2)})(c-1+\sqrt{c(c-2)})^n\right. \nonumber \\ \left.-(c-\sqrt{c(c-2)})(c-1-\sqrt{c(c-2)})^n \right). \label{e.99}
\end{eqnarray}
Also, since $|c-1+\sqrt{c(c-2)}|\not=|c-1-\sqrt{c(c-2)}|$ for $c\not=0,1,2$ and $|c-1+\sqrt{c(c-2)}|\cdot|c-1-\sqrt{c(c-2)}|=1$, we put
\begin{equation}\label{e100}
 Q=\frac{1}{\sqrt{c-2}}(c+\sqrt{c(c-2)})(c-1+\sqrt{c(c-2)})^n,
 \end{equation}
if $|c-1+\sqrt{c(c-2)}|>1$. Alternatively, if $|c-1+\sqrt{c(c-2)}|<1$, ie. $|c-1-\sqrt{c(c-2)}|>1$, we put
\begin{equation}\label{e101} Q=\frac{1}{\sqrt{c-2}}(c-\sqrt{c(c-2)})(c-1-\sqrt{c(c-2)})^n.
\end{equation}
To be more precise, if $\operatorname{Re}(c)>1$ or $\operatorname{Re}(c)=1$
and $\operatorname{Im}(c)>0$, then $Q$ is given by (\ref{e100}) and also
$\sqrt{c(c-2)}=\sqrt{c}\sqrt{c-2}.$ In the other hand, if $0\leq
\operatorname{Re}(c)<1$ or $\operatorname{Re}(c)=1$ and $\operatorname{Im}
(c)<0$, then $\sqrt{c(c-2)}=-\sqrt{c}\sqrt{c-2}$ and $Q$ is defined by
(\ref{e101}). Note that, in both cases, $Q$ can be given by
\begin{equation}\label{e104} Q=\frac{1}{\sqrt{c-2}}(c+\sqrt{c}\sqrt{c-2})(c-1+\sqrt{c}\sqrt{c-2})^n.\end{equation}
Similarly to the previous case, we have
\begin{equation}\label{e92} u'_n=\pm\frac{\epsilon}{2\sqrt{c}}(Q-\frac{2c}{c-2}Q^{-1}),\end{equation}
where
$$ Q^{-1}=\frac{\sqrt{c-2}}{2c}(c-\sqrt{c}\sqrt{c-2})(c-1-\sqrt{c}\sqrt{c-2})^n.$$

The theorem of Baker and W\"{u}stholz (Theorem \ref{l.bw}) will be applied on the form
$$ \Lambda=\log\frac{|Q|}{|P|}. $$

\subsection{Estimates on $ \left|\Lambda\right|$ } \label{s.41}

First, we have to estimate the lower bounds for $|P|$ and $|Q|$. Since $|c|\not=1$, then $|c|\geq\sqrt{2}$ and
\begin{equation}\label{e102} |P|\geq 11.6,\end{equation}
for  $m\geq 2$. Indeed, the inequality (\ref{e102}) follows from the fact that
$$\left\vert \frac{c+\sqrt{c(c+2)}}{\sqrt{c+2}}\right\vert =\left\vert \sqrt
{c}\right\vert \left\vert \frac{\sqrt{c}}{\sqrt{c+2}}+1\right\vert \geq
1\cdot1=1, $$
since $\left\vert \sqrt{c}\right\vert =\sqrt{|c|}\geq1$, $\operatorname{Re}
\left(  \frac{\sqrt{c}}{\sqrt{c+2}}\right)  \geq0$ and
$$\left|c+1+\sqrt{c(c+2)}\right|^2\geq (\left|c\right|+2)^2\geq(\sqrt{2}+2)^2\geq 11.6. $$
Similarly, if $|c|\geq \sqrt{2}$ and $n\geq 2$, then (\ref{e104}) implies
\begin{equation}\label{e103}
|Q|\geq\left\vert \sqrt{c}\right\vert \left\vert \frac{\sqrt{c}}{\sqrt{c-2}
}+1\right\vert \left\vert c-1+\sqrt{c}\sqrt{c-2}\right\vert ^{2}\geq
\left(  \sqrt
{2}+1\right)  ^{2}>5.8. \end{equation}
In the case $|c|\geq 2$, (\ref{e102}) and (\ref{e103}) can be  immediately improved to
$$ |P|\geq 16,\ |Q|\geq 9. $$
Since there are finitely many integers $c$ such that $\sqrt{2}\leq|c|<2$, we can easily obtain  lower bounds for $|P|$ and $|Q|$ assuming $|c|\geq\sqrt{2}$. Indeed,
$$ |P|\geq \min\{16,\min\{P_c\,:\,c\in T\}\},\
|Q|\geq \min\{9,\min\{Q_c\,:\,c\in T\}\}$$
where
$$P_c=\left|\sqrt{\frac{c}{c+2}}\right|\left|\sqrt{c}+\sqrt{c+2}\right|\left|c+1+\sqrt{c(c+2)}\right|^2,$$
$$Q_c=\left|\sqrt{\frac{c}{c-2}}\right|\left|\sqrt{c}\pm\sqrt{c-2}\right|\left|c-1+\sqrt{c}\sqrt{c-2}\right|^2,$$
and
{\small
\begin{equation}\label{e1230}
T=\{1 \pm \sqrt{-1},\pm \sqrt{-2}, \frac{1\pm \sqrt{-7}}{2},1\pm \sqrt{-2}, \pm \sqrt{-3}, \frac{3\pm \sqrt{-3}}{2}, \frac{1\pm\sqrt{-11}}{2}\},\end{equation}}
i.e. $T$ is the set of all integers in $\Q(\sqrt{-D})$ such that $\sqrt{2}\leq|c|<2$, $c\not\in S_c$ and $\on{Re}(c)\geq0$. Finally, we get
$$  |P|\geq 16, |Q|\geq 9, $$
for $|c|\geq\sqrt{2}$ and $n,m\geq 2$ (because $\min\{P_c\,:\,c\in T\}\geq 16$ and $\min\{Q_c\,:\,c\in T\}\geq 9$). Using these bounds, we have to show that the value of $ \left|\log\frac{|Q|}{|P|}\right|$ is small enough.
Assuming that
$$ u_m=\pm u'_n, \ m\geq 2,n\geq 2, $$
relations (\ref{e90}) and (\ref{e92}) imply
\begin{equation}\label{e.111}
 P \pm Q= \pm\frac{2c}{c-2}Q^{-1}+\frac{2c}{c+2}P^{-1}.
\end{equation}
So,
\begin{eqnarray}\label{e111}
||P| -| Q||\leq|P \mp Q|&\leq&\left|\frac{2c}{c-2}\right||Q|^{-1}+\left|\frac{2c}{c+2}\right||P|^{-1}\\ \notag
&<& 2\cdot 5\cdot\frac{1}{9}+2\cdot 1\cdot\frac{1}{16} =1.24.
\end{eqnarray}
(Note that $\left|\frac{c}{c-2}\right|\leq 5$ for $c\in\Z_M$, $c\not=2$, $\on{Re}(c)\geq0$, and $\left|\frac{c}{c+2}\right|\leq 1$ for $c\in\Z_M$ and $\on{Re}(c)\geq0$.)
Since,
$$ \frac{||P| -| Q||}{|P|}<1.24|P|^{-1}<1, $$
we have
$$ \left|\log\frac{|Q|}{|P|}\right|=\log\left|1-\frac{|P|-|Q|}{|P|}\right|\leq \frac{||P| -| Q||}{|P|}+\left(\frac{||P| -| Q||}{|P|}\right)^2.$$
Also, the inequality $||P| -| Q||<1.24$ implies that
$$|P|<| Q|+1.24\leq| Q|+1.24\frac{| Q|}{9}<1.14| Q|,$$
or equivalently
$$|Q|^{-1}<1.14| P|^{-1}.$$
 By putting that into (\ref{e111}), we get
$$ ||P| -| Q||<\left|\frac{2c}{c-2}\right|\cdot 1.14| P|^{-1}+\left|\frac{2c}{c+2}\right||P|^{-1}<13.4|P|^{-1}.$$
Finally,
$$ \left|\log\frac{|Q|}{|P|}\right|<13.4|P|^{-2}+(13.4|P|^{-2})^2 \leq(13.4+(13.4\frac{1}{16})^2)|P|^{-2}<14.11|P|^{-2}.$$
Furthermore it can be shown that
$$  |\Lambda|=\left|\log\frac{|Q|}{|P|}\right|<14.11|P|^{-2}<14.11\cdot\left(  \sqrt{2}+2\right)  ^{-2m}<3^{-m},$$
for $|c|\geq\sqrt{2}$ and $m,n\geq 2$.
Similarly we find
$$ \frac{||P|-|Q||}{|Q|}<1.24|Q|^{-1}<1 $$
and
\begin{eqnarray*}
|\Lambda | &=&\left\vert \log 
\frac{|P|}{|Q|}\right\vert =\log \left\vert 1-\frac{|P|-|Q|}{|Q|}\right\vert
\leq \frac{||P|-|Q||}{|Q|}+\left( \frac{||P|-|Q||}{|Q|}\right) ^{2} \\
&<&14|Q|^{-2}<14\cdot \left( \sqrt{2}+1\right) ^{-2n}<\left( 1.55\right)
^{-n}.
\end{eqnarray*}

It remains to show that there are no solutions in cases when $m=1$ or $n=1.$
For $c\in\Z_M$, $|c|\geq\sqrt{2}$ and $\operatorname{Re}(c)\geq0$, we have
$$|u_{1}|=|2c+1|\leq2|c|+1,\ \ |u_{1}^{\prime}|=|2c-1|\leq2|c|+1$$
and
$$
|u_{m}|>(2\sqrt{1+|c|^{2}}-1)^{m},\ \ \ |u_{n}^{\prime}|>(2\sqrt{1+|c|^{2}
}-1)^{n-1}(2|c|-1),$$
for $m,n\geq2$ (where last inequalities are obtained similarly to those in
Lemma \ref{l.13}). Since,
\[
|u_{1}|\leq2|c|+1<(2\sqrt{1+|c|^{2}}-1)(2|c|-1)\leq|u_{n}^{\prime}|,\text{
if\ }n\geq2
\]
\[
|u_{1}^{\prime}|\leq2|c|+1<(2\sqrt{1+|c|^{2}}-1)^{2}\leq|u_{m}|,\text{
if\ }m\geq2
\]
we conclude that the equations $u_{1}=\pm u_{n}^{\prime}$ and $u_{1}^{\prime
}=\pm u_{m}$ have no solution for $m,n\geq2$, i.e. for $m,n\geq1$ (because
$u_{1}\not =\pm u_{1}^{\prime}$).

\subsection{The condition $\Lambda\not=0$}

Without proving that $\Lambda\not=0$, i.e. $|P|\not=|Q|$, we cannot apply Theorem \ref{l.bw}. This proof is rather complicated and involves several cases.

First, we show that $P\not=\pm Q$. 
 Let us assume that $P=\pm Q$. According to (\ref{e.111}),  the following possibilities may occur:
$$\ds{\frac{c}{c^2-4}=0},$$
 which is obviously not possible ($c\not=0,\pm2$), and
 $$\ds{P^2=\frac{2c}{c^2-4}}. $$
In Section \ref{s.41}, we have shown that $|P|^2\geq 16^2$. Since
$$ \left|\frac{2c}{c^2-4}\right|\leq\frac{2|c|}{||c|^2-4|}\leq\frac{2\cdot\sqrt{5}}{5-4}<5, $$
for $|c|\geq\sqrt{5}$ 
$$ \left|\frac{2c}{c^2-4}\right|\leq \max\{\left|\frac{2c}{c^2-4}\right|\,:\,c\in T_1\}<1, \ c\in T_{1}, $$
where $T_{1}=\{c\in\mathbb{Z}_{M}\backslash S_{c}:\,\sqrt{2}\leq|c|\leq2\}$, ie.
$$
T_{1}=T\cup
\left\{  \pm2\sqrt{-1},\frac{3\pm\sqrt{-7}}{2},\frac{1\pm\sqrt{-15}}{2}%
,1\pm\sqrt{-3}\right\}$$
and $T$ is given in (\ref{e1230}). Hence, a contradiction is reached.\\

Before presenting other cases, let us take a closer look at $|P|$ and $|Q|$ from an algebraic point of view. According to (\ref{e.100p}), we have
\begin{equation}\label{Pc}
\frac{ P}{\sqrt{c}}=\frac{c+\sqrt{c(c+2)}}{\sqrt{c(c+2)}}(c+1+\sqrt{c(c+2)})^m=a+b\alpha=a+\frac{b_1}{c+2}\alpha,\end{equation}
where $\alpha=\sqrt{c(c+2)}$ is an algebraic integer in $\Q(\sqrt{-D})$ and $a,b_1\in \Z_M$. Similarly, (\ref{e100}) and (\ref{e101}) imply that
\begin{equation}\label{Qc}
\frac{ Q}{\sqrt{c}}=d+e\beta=d+\frac{e_1}{c-2}\beta,\end{equation}
where $\beta=\sqrt{c(c-2)}$ is an algebraic integer in $\Q(\sqrt{-D})$ and $d,e_1\in \Z_M$.  It follows straight away that
\begin{eqnarray*}
u_m&=&\frac{\epsilon}{2}(a+b\alpha+a-b\alpha)=\epsilon a,\\
u'_n&=&\frac{\epsilon}{2}(d+e\beta+d-e\beta)=\epsilon d ,
\end{eqnarray*}
where we have used the  explicit expressions (\ref{e.98}) and (\ref{e.99}) for $u_m$ and $u'_n$ . Since $u_m=\pm u'_n$, we get
$$ a=\pm d.$$
Note that $a\not=0$, $d\not=0$, because $|u_m|,|u'_n|>0$ for $m,n\geq2$.  We have
$$ \left|\frac{ P}{\sqrt{c}} \right|^2=(a+b\alpha)(\ol{a}+\ol{b}\ol{\alpha})=|a|^2+(\ol{a}b)\alpha+(a\ol{b})\ol{\alpha}+|b|^2|\alpha|^2$$
and analogously
$$ \left|\frac{Q}{\sqrt{c}} \right|^2=(d+e\alpha)(\ol{d}+\ol{e}\ol{\beta})=|d|^2+(\ol{d}e)\beta+(d\ol{e})\ol{\beta}+|e|^2|\beta|^2.$$
If the algebraic extension $\Q(\sqrt{-D})(\alpha,\ol{\alpha},\beta,\ol{\beta})$ is considered as a vector space over $\Q(\sqrt{-D})$, then $\{1,\alpha,\ol{\alpha},|\alpha|^2,\beta,\ol{\beta},|\beta|^2\}$ is its generating set and $\left|\frac{ P}{\sqrt{c}} \right|^2$ and $\left|\frac{ Q}{\sqrt{c}} \right|^2$ are  the elements of the vector subspaces $span(1,\alpha,\ol{\alpha},|\alpha|^2)$ and $span(1,\beta,\ol{\beta},|\beta|^2)$, respectively. Before continuing with the proof, we establish   the following useful claims:
\begin{description}
\item[(i)] \emph{If $c\not\in\{0,\pm 1,\pm 2\}$, then $\alpha,\beta\not\in \Q(\sqrt{-D})$.}

 Indeed, we can show that $\alpha\in\mathbb{Q}(\sqrt{-D})$ if and only if
$c=0,-1,-2$. Let $\alpha\in\mathbb{Q}(\sqrt{-D}).$ Note that $\alpha
\in\mathbb{Q}(\sqrt{-D})$ if and only if $c(c+2)=t^{2}$ for some
$t\in\mathbb{Z}_{M}$. Therefore, $c=-1\pm\sqrt{t^{2}+1}$, where $t^{2}%
+1=s^{2}$ for some $s\in\mathbb{Z}_{M}$. Note that $t,s,c\in\mathbb{Z}_{M}$
and $t\pm s$ are units in $\mathbb{Z}_{M}$. It is easy to check that only
possibilities are $c=0,-1,-2$.

It can be proved similarly that  $\beta\in \Q(\sqrt{-D})$ if and only if $c = 0,1,2$.

\item[(ii)] \emph{If $B_{1}$ is a basis of the subspace $span(1,\alpha
,\overline{\alpha},|\alpha|^{2})$, then $B_{1}=\{1,\alpha,\overline{\alpha
},|\alpha|^{2}\}$ or $B_{1}=\{1,\alpha\}$. Set $\{1,\alpha\}$ is a basis of
$span(1,\alpha,\overline{\alpha},|\alpha|^{2})$ if and only if $\overline
{\alpha}=K\alpha$, $K\in\mathbb{Q}(\sqrt{-D})$. The analogous statement is
true for a basis of $span(1,\beta,\overline{\beta},|\beta|^{2})$.}

According to (i), it is obvious that $\{1,\alpha\}$ is a linearly independent
set. Let $\overline{\alpha}=A+C\alpha$, for $A,C\in\mathbb{Q}(\sqrt{-D})$. By
squaring it, we obtain $2AC\alpha=\overline{\alpha}^{2}-A^{2}-\alpha^{2}%
C^{2}\in\mathbb{Q}(\sqrt{-D})$. Since $\alpha\not \in \mathbb{Q}(\sqrt{-D})$,
we have that $AC=0$. If $C=0$, then $\overline{\alpha}=A\in\mathbb{Q}%
(\sqrt{-D})$, a contradiction. If $A=0$, then $\overline{\alpha}=C\alpha$ and
$|\alpha|^{2}=C\alpha^{2}\in\mathbb{Q}(\sqrt{-D})$ imply $B_1=\{1,\alpha\}$.

If $\{1,\alpha,\overline{\alpha}\}$ is a linearly independent set and
$|\alpha|^{2}=A+C\alpha+E\overline{\alpha}$ for $A,C,E\in\mathbb{Q}(\sqrt
{-D})$, then by squaring it we get
\[
\underbrace{|\alpha|^{4}}_{\in\mathbb{Q}(\sqrt{-D})}=\underbrace{A^{2}%
+C^{2}\alpha^{2}+E^{2}\overline{\alpha}^{2}}_{\in\mathbb{Q}(\sqrt{-D}%
)}+2AC\alpha+2AE\overline{\alpha}+2CE\underbrace{(A+C\alpha+E\overline{\alpha
})}_{|\alpha|^{2}},
\]
a linear combination of $1,\alpha,\overline{\alpha}$. So, $C(A+E)=0$ and
$E(A+C)=0$. If $A=C=0$, then $|\alpha|^{2}=E\overline{\alpha}$ implies
$\alpha\in\mathbb{Q}(\sqrt{-D})$, a contradiction. Also, other two cases end
with a contradiction ( $A=E=0$ implies $\overline{\alpha}\in\mathbb{Q}%
(\sqrt{-D})$ and $C=E=0$ implies that $\{\alpha,\overline{\alpha}\}$ is a
linearly dependent set).

\item[(iii)] \emph{Let }$c\not \in \{0,\pm1,\pm2\}.$\emph{ If }$\beta\in
span(1,\alpha,\overline{\alpha},|\alpha|^{2})$\emph{, then }$B_{1}%
=\{1,\alpha,\overline{\alpha},|\alpha|^{2}\}$\emph{ is a basis of the subspace
}$span(1,\alpha,\overline{\alpha},|\alpha|^{2})$\emph{ and }$\beta
=K\overline{\alpha}$\emph{ or }$\beta=K\left\vert \alpha\right\vert ^{2}%
$\emph{, for some }$K\in\mathbb{Q}(\sqrt{-D})$\emph{. The analogous statement
is true if }$\alpha\in$\emph{$span(1,\beta,\overline{\beta},|\beta|^{2})$.}

\item Let $\beta\in span(1,\alpha,\overline{\alpha},|\alpha|^{2}).$ Obviously,
this implies $\overline{\beta},$ $|\beta|^{2}\in span(1,\alpha,\overline
{\alpha},|\alpha|^{2})$ too. If we assume that $\beta=K\alpha$ for some
$K\in\mathbb{Q}(\sqrt{-D})$, then $K=\pm\frac{\sqrt{c^{2}-4}}{c-2}$.
Therefore, $c^{2}-4=r^{2}$ for some $r\in\mathbb{Z}_{M}$. Since $c,r\in
\mathbb{Z}_{M}$ and $|c\pm r|\leq4$, by checking all possibilities, we find
$c=0,\pm1,-2.$ (Similarly, if $\beta=K\alpha$ for $K\in\mathbb{Q}(\sqrt{-D})$,
then $c=0,\pm1,2$).

\item If $B_{1}=\{1,\alpha\}$ is basis of subspace $span(1,\alpha
,\overline{\alpha},|\alpha|^{2}),$ then $\beta=L+K\alpha$ some $L,K\in
\mathbb{Q}(\sqrt{-D}).$ Then, by squaring it, it is easy to see $\beta
=L\in\mathbb{Q}(\sqrt{-D})$ or $\beta=K\alpha,$ which is impossible.

\item If $B_{1}=\{1,\alpha,\overline{\alpha},|\alpha|^{2}\}$ is basis of
subspace $span(1,\alpha,\overline{\alpha},|\alpha|^{2}),$ then $\beta
=L+K\alpha+K^{\prime}\overline{\alpha}+K^{\prime\prime}\left\vert
\alpha\right\vert ^{2}$ for some $L,K,K^{\prime},K^{\prime\prime}\in
\mathbb{Q}(\sqrt{-D}).$ Similarly as before, we obtain $\beta\in
\mathbb{Q}(\sqrt{-D})$ or $\beta=K\alpha\ $or $\beta=K^{\prime}\overline
{\alpha}$ or $\beta=K^{\prime\prime}\left\vert \alpha\right\vert ^{2}.$
Therefore, we might have $\beta=K^{\prime}\overline{\alpha}$ or $\beta
=K^{\prime\prime}\left\vert \alpha\right\vert ^{2}$, since first two cases are impossible.\ 
\end{description}

Furthermore,
\[
\left|  \frac{ P}{\sqrt{c}} \right|  ^{2}-\left|  \frac{ Q}{\sqrt{c}} \right|
^{2}\ \in\ V=span(1,\alpha,\overline{\alpha},|\alpha|^{2},\beta,\overline
{\beta},|\beta|^{2}).
\]
There are several possibilities for choosing a basis $B$ for $V$ from its
generating set $\{1,\alpha,\overline{\alpha},|\alpha|^{2},\beta,\overline
{\beta},|\beta|^{2}\}$:

\begin{description}
\item[(a)] $B=\{1\}$. This happens if and only if $\alpha,\beta\in
\mathbb{Q}(\sqrt{-D})$. So, this is not possible.

\item[(b)] $B=\{1,\alpha\}$ (or $B=\{1,\beta\}$). This is also not possible.
Indeed, in this case $B_{1}=\{1,\alpha\}$ is basis of subspace $span(1,\alpha
,\overline{\alpha},|\alpha|^{2})$ and $\beta\in span(1,\alpha,\overline
{\alpha},|\alpha|^{2}),$ which contradicts (iii).

\item[(c)] $B=\{1,\alpha,\overline{\alpha},|\alpha|^{2}\}$ (or $B=\{1,\beta
,\overline{\beta},|\beta|^{2}\}$). In this case $B_{1}=\{1,\alpha
,\overline{\alpha},|\alpha|^{2}\}$ is basis of subspace $span(1,\alpha
,\overline{\alpha},|\alpha|^{2})$ and $\beta\in span(1,\alpha,\overline
{\alpha},|\alpha|^{2}).$ This case implies that $\beta=K\overline{\alpha}$ or
$\beta=K|\alpha|^{2}$ for $K\in\mathbb{Q}(\sqrt{-D})$ according to (iii).

\item[(d)] $B=\{1,\alpha,\beta\}$. This implies that $\overline{\alpha
}=K\alpha$ and $\overline{\beta}=L\beta$, for $K,L\in\mathbb{Q}(\sqrt{-D})$.

\item[(e)] $B=\{1,\alpha,\beta,\overline{\beta},|\beta|^{2}\}$ (or
$B=\{1,\alpha,\overline{\alpha},|\alpha|^{2},\beta\}$). Here, we have
$\overline{\alpha}=K\alpha$ for $K\in\mathbb{Q}(\sqrt{-D})$ (or $\overline
{\beta}=K\beta$ for $K\in\mathbb{Q}(\sqrt{-D})$)

\item[(f)] $B=\{1,\alpha,\overline{\alpha},|\alpha|^{2},\beta,\overline{\beta
},|\beta|^{2}\}$.
\end{description}

In what follows, we show that $|P|\not =|Q|$ in each of cases (c) to (f)
unless $\operatorname{Re}(c)=0.$ Assume that $|P|=|Q|$, ie.
\begin{equation}
0=\left\vert \frac{P}{\sqrt{c}}\right\vert ^{2}-\left\vert \frac{Q}{\sqrt{c}%
}\right\vert ^{2}=(\overline{a}b)\alpha+(a\overline{b})\overline{\alpha
}+|b|^{2}|\alpha|^{2}-(\overline{d}e)\beta-(d\overline{e})\overline{\beta
}-|e|^{2}|\beta|^{2}. \label{e123}%
\end{equation}

\emph{\textbf{Case (f)}}: Let $B=\{1,\alpha,\overline{\alpha},|\alpha
|^{2},\beta,\overline{\beta},|\beta|^{2}\}$ basis $B$ for $V.$ Since the set
$\{\alpha,\overline{\alpha},|\alpha|^{2},\beta,\overline{\beta},|\beta|^{2}\}$
is linearly independent, all coefficients have to be zero:
\[
\overline{a}b=a\overline{b}=|b|^{2}=\overline{d}e=d\overline{e}=|e|^{2}=0.
\]
This implies $b=e=0$ and
\[
P=a\sqrt{c}=\pm d\sqrt{c}=\pm Q,
\]
which is not possible.

\emph{\textbf{Case (e)}}: The assumption is that the set $B=\{1,\alpha
,\beta,\overline{\beta},|\beta|^{2}\}$ form a basis for $V$. In this case we
know that $\overline{\alpha}=K\alpha$ and $|\alpha|^{2}=K\alpha^{2}$ for
$K\in\mathbb{Q}(\sqrt{-D})$. Obviously, $K\not =0$. So, (\ref{e123}) imply
\[
(|b|^{2}K\alpha^{2})1+(\overline{a}b+a\overline{b}K)\alpha-(\overline
{d}e)\beta-(d\overline{e})\overline{\beta}-|e|^{2}|\beta|^{2}=0.
\]
The coefficients must be zero:
\[
|b|^{2}\underbrace{K\alpha^{2}}_{\not =0}=\overline{a}b+a\overline
{b}K=\overline{d}e=d\overline{e}=|e|^{2}=0.
\]
Hence, $b=e=0$ and
\[
P=a\sqrt{c}=\pm d\sqrt{c}=\pm Q,
\]
which is not possible. Similarly, we obtain a contradiction, if we assume
$B=\{1,\alpha,\overline{\alpha},|\alpha|^{2},\beta\}$ is a basis for $V.$

\emph{\textbf{Case (d)}}: The set $B=\{1,\alpha,\beta\}$ form a basis for $V$.
This is a situation when
\[
\overline{\alpha}=K\alpha,\ |\alpha|^{2}=K\alpha^{2},\ \overline{\beta}%
=L\beta,\ |\beta|^{2}=L\beta^{2},
\]
for $K,L\in\mathbb{Q}(\sqrt{-D})$ and $K,L\not =0$. Furthermore, $K,L$ are
units in $\mathbb{Q}(\sqrt{-D})$, ie. $|K|=|L|=1$. Implementing that into
(\ref{e123}) we get
\[
(|b|^{2}K\alpha^{2}-|e|^{2}L\beta^{2})1+(\overline{a}b+a\overline{b}%
K)\alpha-(\overline{d}e+d\overline{e}L)\beta=0,
\]
and
\[
|b|^{2}K\alpha^{2}=|e|^{2}L\beta^{2},\ \overline{a}b=-a\overline
{b}K,\ \overline{d}e=-d\overline{e}L.
\]
Assume that $b,e\not =0$. Since $a,d\not =0$, substituting
\[
K=-\frac{\overline{a}b}{a\overline{b}},\ L=-\frac{\overline{d}e}{d\overline
{e}},
\]
we get
\[
|b|^{2}\frac{\overline{a}b}{a\overline{b}}\,\alpha^{2}=|e|^{2}\frac
{\overline{d}e}{d\overline{e}}\,\beta^{2}\ \Leftrightarrow\ \left(
\frac{\overline{a}b}{|a|}\right)  ^{2}\alpha^{2}=\left(  \frac{\overline{d}%
e}{|d|}\right)  ^{2}\beta^{2}.
\]
Also since, $a=\pm d$, we have
\[
b^{2}\alpha^{2}=e^{2}\beta^{2}.
\]
So,
\[
(b\alpha-e\beta)(b\alpha+e\beta)=0
\]
and this leads to $b=e=0$, because $\alpha,\beta$ are linearly independent,
which gives again $P=\pm Q$. A contradiction!

\emph{\textbf{Case (c)}}: Recall that $\{1,\alpha,\overline{\alpha}%
,|\alpha|^{2}\}$ forms a basis of $V$ and $\beta=K\overline{\alpha}$ or
$\beta=K|\alpha|^{2}$ for $K\in\mathbb{Q}(\sqrt{-D})$.

If $\beta=K|\alpha|^{2}$, then $\overline{\beta}=\overline{K}|\alpha|^{2}$ and
$|\beta|^{2}=|K|^{2}|\alpha|^{4}\in\mathbb{Q}(\sqrt{-D})$. So, (\ref{e123})
implies that
\[
(-|e|^{2}|K|^{2}|\alpha|^{4})\cdot1+(\overline{a}b)\alpha+(a\overline
{b})\overline{\alpha}+(|b|^{2}-\overline{d}eK-d\overline{e}\overline
{K})|\alpha|^{2}=0.
\]
Therefore,
\[
|e|^{2}|K|^{2}|\alpha|^{4}=0,\ \overline{a}b=0,\ |b|^{2}-\overline
{d}eK-d\overline{e}\overline{K}=0.
\]
Evidently $e=b=0$ which imply $P=\pm Q$, a contradiction.

If $\beta=K\overline{\alpha}$, then $\overline{\beta}=\overline{K}\alpha$ and
$|\beta|^{2}=|K|^{2}|\alpha|^{2}$. Notice that $\overline{\beta}\not =L\beta$
for all $L\in\mathbb{Q}(\sqrt{-D})$. (If $\overline{\beta}=L\beta$, then
$\beta=L^{-1}K\alpha,$ which is not possible.) According to (\ref{e123}) we
have
\[
(\overline{a}b-d\overline{e}\overline{K})\alpha+(a\overline{b}-\overline
{d}eK)\overline{\alpha}+(|b|^{2}-|e|^{2}|K|^{2})|\alpha|^{2}=0,
\]
and
\[
a\overline{b}-\overline{d}eK=0,\ |b|^{2}-|e|^{2}|K|^{2}=0.
\]
Therefore,%
\begin{equation}
K=\frac{a\overline{b}}{\overline{d}e}=\pm\frac{a\overline{b}}{\overline{a}e}.
\label{K}%
\end{equation}
From (\ref{Pc}) and (\ref{Qc}) we obtain
\begin{equation}\label{abe1}
a^{2}-b^{2}c\left(  c+2\right)  =\frac{2}{c+2},\text{ \ \ \ \ }a^{2}
-e^{2}c\left(  c-2\right)  =-\frac{2}{c-2} 
\end{equation}
and
\begin{equation}
cb_{1}^{2}-\left(  c+2\right)  a^{2}=-2,\text{ \ \ \ }\left(  c-2\right)
a^{2}-ce_{1}^{2}=-2, \label{abe}
\end{equation}
which again implies
\begin{equation}
e^{2}c\left(  c-2\right)  -b^{2}c\left(  c+2\right)  =\frac{4c}{c^{2}-4}\text{
\ and \ }\left(  c+2\right)  e_{1}^{2}-\left(  c-2\right)  b_{1}^{2}=4.
\label{eb}%
\end{equation}
Equation (\ref{K}) implies $\left\vert e\beta\right\vert =\left\vert
b\alpha\right\vert ,$ which again implies $\left\vert e_{1}^{2}\left(
c+2\right)  \right\vert =\left\vert b_{1}^{2}\left(  c-2\right)  \right\vert
.$ Let%
\[
X=\left(  c+2\right)  e_{1}^{2}\text{ \ \ and \ \ }Y=\left(  c-2\right)
b_{1}^{2}.
\]
Therefore, we have $X-Y=4$ and $\left\vert X\right\vert =\left\vert
Y\right\vert ,$ which implies $\operatorname{Re}X=2,$ $\operatorname{Re}Y=-2.$
On the other hand, from (\ref{K}) and (\ref{eb}) we obtain
\begin{equation}
\frac{c^{2}-4}{c\cdot\overline{a}^{2}}\left(  \overline{b}^{2}a^{2}%
\overline{\alpha}^{2}-\overline{a}^{2}b^{2}\alpha^{2}\right)  =4. \label{2Imz}%
\end{equation}
Since $\overline{b}^{2}a^{2}\overline{\alpha}^{2}-\overline{a}^{2}b^{2}%
\alpha^{2}=2\operatorname{Im}\left(  a\overline{b}\overline{\alpha}\right)
^{2}i,$ equation (\ref{2Imz}) implies%
\begin{equation}
\operatorname{Re}\left(  \frac{c^{2}-4}{c\overline{a}^{2}}\right)  =0.
\label{Re}%
\end{equation}
The condition (\ref{Re}) is equivalent to the condition $\operatorname{Re}%
\left(  \frac{\left(  c^{2}-4\right)  a^{2}}{c}\right)  =0.$ On the other
hand, from (\ref{abe}) and $\operatorname{Re}\left(  \left(  c+2\right)
e_{1}^{2}\right)  =2,$ we obtain
\begin{eqnarray*}
\operatorname{Re}\left(  \frac{\left(  c^{2}-4\right)  a^{2}}{c}\right)
&=&\operatorname{Re}\left(  \frac{\left(  c+2\right)  \left(  ce_{1}
^{2}-2\right)  }{c}\right)  \\ &=&\operatorname{Re}\left(  -2-\frac{4}{c}+e_{1}
^{2}\left(  c+2\right)  \right)  =-\operatorname{Re}\left(  \frac{4}%
{c}\right)  =0,\end{eqnarray*}
which again implies $\operatorname{Re}c=0$, ie. $c=vi,$ $v\in
\mathbb{Z}\left(  \sqrt{D}\right)  ,$ $v\neq0,\pm1.$ In general we have
$\overline{\sqrt{z}}=\sqrt{\overline{z}},\ $if $z\in\mathbb{C}\backslash
\mathbb{R}^{-}\ $and$\ \overline{\sqrt{z}}=-\sqrt{\overline{z}},\ $%
if\ $z\in\mathbb{R}^{-}.$ Since, we have $\beta=\sqrt{vi\left(
vi-2\right)  }$, $\alpha=\sqrt{vi\left(  vi+2\right)  }$ and $\overline{\left(
vi-2\right)  vi}=vi\left(  vi+2\right)  \notin\mathbb{R}^{-}$, 
then $\beta=\overline{\alpha}$ ie. $K=1$. On the other hand, from (\ref{2Imz})
we obtain that $\operatorname{Re}\left(  \frac{-4iv}{v^{2}+4}\left(
\overline{a}\right)  ^{2}\right)  =0,$ which again implies $\overline{a}=-a$
or $\overline{a}=a.$ Therefore, since\ $K=1$ from (\ref{K}), we obtain
$e=\pm\overline{b}$ and distinguish four cases:

\begin{enumerate}
\item If $\overline{a}=a=d$, then $e=\frac{a\overline{b}}{\overline{d}%
}=\overline{b}$ and $\frac{P}{\sqrt{c}}=a+b\alpha,$ $\frac{Q}{\sqrt{c}%
}=a+\overline{b}\overline{\alpha}$;
\item $\overline{a}=-a,$ $a=d$, then $e=\frac{a\overline{b}}{\overline{d}%
}=-\overline{b}$ and $\frac{P}{\sqrt{c}}=a+b\alpha,$ $\frac{Q}{\sqrt{c}%
}=a-\overline{b}\overline{\alpha}$;
\item $\overline{a}=a,$ $a=-d$, then $e=\frac{a\overline{b}}{\overline{d}%
}=-\overline{b}$ and $\frac{P}{\sqrt{c}}=a+b\alpha,$ $\frac{Q}{\sqrt{c}%
}=-a-\overline{b}\overline{\alpha}$;
\item $\overline{a}=-a,$ $a=-d$, then $e=\frac{a\overline{b}}{\overline{d}%
}=\overline{b}$ and $\frac{P}{\sqrt{c}}=a+b\alpha,$ $\frac{Q}{\sqrt{c}%
}=-a+\overline{b}\overline{\alpha}$.
\end{enumerate}
Note that each of the above cases implies that $|P|=|Q|$ and therefore $\Lambda=0$. In what follows we show that in this particular case the equation  $u_m=\pm u_n'$, $m,n>0$,  has no solution.

First we will show that $\overline{a}=\pm a$ imply $b\alpha\neq\pm\overline
{b}\overline{\alpha}.$ It is enough to show $\operatorname{Im}\left(
b\alpha\right)  ^{2}\neq0.$ Suppose $\operatorname{Im}\left(  b\alpha\right)
^{2}=0.$ Then, from (\ref{abe1}) we have
\[
\left(  b\alpha\right)  ^{2}=b^{2}c\left(  c+2\right)  =a^{2}-\frac{2}{vi+2}.
\]
Since,\ we have $\operatorname{Im}\left(  b\alpha\right)  ^{2}%
=\operatorname{Im}a^{2}=0$, then $\operatorname{Im}\frac{2}{vi+2}=0\ $which
again implies $v=0,$ a contradiction.

From (\ref{e.111}) we obtain%
\begin{align}
\frac{P}{\sqrt{c}}-\frac{Q}{\sqrt{c}}  &  =-\frac{2}{\left(  c-2\right)
}\cdot\frac{\sqrt{c}}{Q}+\frac{2}{\left(  c+2\right)  }\cdot\frac{\sqrt{c}}%
{P},\ \ \text{if \ }a=d,\label{PQ-}\\
\frac{P}{\sqrt{c}}+\frac{Q}{\sqrt{c}}  &  =\frac{2}{\left(  c-2\right)  }%
\cdot\frac{\sqrt{c}}{Q}+\frac{2}{\left(  c+2\right)  }\cdot\frac{\sqrt{c}}%
{P},\ \ \text{if \ }a=-d. \label{PQ+}%
\end{align}
If $\overline{a}=a=d$, then (\ref{PQ-}) imply
\[
b\alpha-\overline{b}\overline{\alpha}=\frac{2}{2-vi}\cdot\frac{1}{\overline
{a}+\overline{b}\overline{\alpha}}+\frac{2}{2+vi}\cdot\frac{1}{a+b\alpha}.
\]
Since $\operatorname{Re}\left(  b\alpha-\overline{b}\overline{\alpha}\right)
=0$ and $\operatorname{Im}\left(  \frac{2}{2-vi}\cdot\frac{1}{\overline
{a}+\overline{b}\overline{\alpha}}+\frac{2}{2+vi}\cdot\frac{1}{a+b\alpha
}\right)  =0,$ we obtain $b\alpha=\overline{b}\overline{\alpha},$ a
contradiction. Similarly, we obtain contradiction in other three cases.  
Analogous results are obtained if  $B=\{1,\beta,\overline{\beta}
,|\beta|^{2}\}$ is a basis for $V.$

Note that if $c=vi$, $v\in\mathbb{Z}[\sqrt{D}],$ $v\neq0,\pm1,$ then
$\beta=\overline{\alpha}$ and $\{1,\alpha,\overline{\alpha},|\alpha|^{2}\}$
forms a basis of $V=\mathbb{Q}(\sqrt{-D})(\alpha,\overline{\alpha}%
,\beta,\overline{\beta})$, that is Case (c). Therefore we have proved the
following assertion.

\begin{proposition}
\label{PropL0}Let $c\notin S_{c}\ $and $\Lambda=\log\frac{|P|}{|Q|}.$ Then

\begin{description}
\item[i)] $\Lambda\not =0$ if and only if $\operatorname{Re}(c)\not =0$. 

\item[ii)] If $\operatorname{Re}(c)=0$, then the equation $u_{m}=\pm
u_{n}^{\prime}$ has no solution for $m,n>0.$
\end{description}
\end{proposition}

\subsection{A reduction procedure}
We consider following linear form in logarithms of algebraic numbers,
$$ \Lambda=\log|Q|-\log|P|=n\log\eta
-m\log\vartheta +\log\xi,$$
where
$$\eta=|c-1+\sqrt{c}\sqrt{c-2}|,\ \vartheta=|c+1+\sqrt{c(c+2)}|,\ \xi=\left|\frac{\sqrt{c+2}(\sqrt{c}+\sqrt{c-2})}{\sqrt{c-2}(\sqrt{c}+\sqrt{c+2})}\right|.$$
First we have to calculate  the standard logarithmic Weil height of $\eta$, $\vartheta$ and $\xi$. Since, the standard logarithmic Weil height $h(\alpha)$ is bounded by 
$$ h(\alpha)\leq\frac{1}{k}\log\left( a_0\prod_{i=1}^n\max\{1,|\alpha^{(i)}|\}\right),$$
where the algebraic number $\alpha$ is a root of $a_0\prod_{i=1}^k(x-\alpha^{(i)})$. 
Note that $\eta$, $\vartheta$ and $\xi$  are roots of the following polynomials
$$ p_1(x)=1 -4(1 - c - 4 \ol{c} + c \ol{c})x^2 +  (6 - 8 c + 4 c^2 - 8
\ol{c} + 4
\ol{c}^2)x^4 -4(1 - c - 4 \ol{c} + c \ol{c})x^6+ x^8,  $$
$$ p_2(x)=1 -4(1+ c+\ol{c}+c\ol{c})x^2 +  (6 +8 c + 4 c^2 + 8
\ol{c} + 4
\ol{c}^2)x^4 -4(1+ c+\ol{c}+c\ol{c})x^6+ x^8,  $$
$p_3(x)=$ {\tiny$\frac{(-2 + c)^8 (-2 +\ol{c}^8)}{(2 + c)^8 (2 +\ol{c})^8}-4\frac{ (-2 + c)^8 (-2 +\ol{c}^8)}{(2 + c)^7 (2 +\ol{c})^7}x^2-24\frac{(-2 + c)^7 (-2 +
\ol{c}^7 (-5 + c^2 +\ol{c}^2)}{(2 + c)^7 (2 +\ol{c})^7}x^4$

$
+4 \frac{ (-2 + c)^7 (-2 +\ol{c})^7 (-35 + 4 c^2 + 4\ol{c}^2)}{(2 + c)^6 (2 +
\ol{c})^6}x^6+4\frac{(-2 + c)^6 (-2 +\ol{c})^6 (455 - 116 c^2 + 4 c^4 - 116
\ol{c}^2 + 44 c^2\ol{c}^2 + 4\ol{c}^4)}{(2 + c)^6 (2 +\ol{c})^6}x^8$

$
 -4\frac{ (-2 + c)^6 (-2 +\ol{c})^6 (273 - 36 c^2 - 36\ol{c}^2 + 16 c^2
\ol{c}^2)}{(2 + c)^5 (2 +\ol{c})^5}x^{10}$

$ -8\frac{ (-2 + c)^5 (-2 +
\ol{c})^5 (-1001 + 253 c^2 - 8 c^4 + 253\ol{c}^2 - 72 c^2
\ol{c}^2 + 8 c^4\ol{c}^2 - 8\ol{c}^4 + 8 c^2
\ol{c}^4)}{(2 + c)^5 (2 +\ol{c})^5}x^{12}$

$
+4\frac{ (-2 + c)^5 (-2 +\ol{c})^5 (-715 + 88 c^2 + 88
\ol{c}^2 + 16 c^2\ol{c}^2)}{(2 + c)^4 (2 +\ol{c})^4}x^{14}
$

 $2 \frac{(-2 + c)^4 (-2 +
\ol{c})^4 (6435 - 1584 c^2 + 48 c^4 -
   1584\ol{c}^2 + 16 c^2\ol{c}^2 + 64 c^4\ol{c}^2 + 48\ol{c}^4 + 64 c^2
\ol{c}^4)}{(2 + c)^4 (2 +\ol{c})^4}x^{16}$

$+ 4\frac{ (-2 + c)^4 (-2 +
\ol{c})^4 (-715 + 88 c^2 + 88
\ol{c}^2 + 16 c^2
\ol{c}^2)}{(2 + c)^3 (2 +
\ol{c})^3}x^{18}$

$ -8\frac{ (-2 + c)^3 (-2 +
\ol{c})^3 (-1001 + 253 c^2 - 8 c^4 + 253
\ol{c}^2 - 72 c^2
\ol{c}^2 + 8 c^4
\ol{c}^2 - 8
\ol{c}^4 + 8 c^2
\ol{c}^4)}{(2 + c)^3 (2 +
\ol{c})^3}x^{20}$

$-4\frac{ (-2 + c)^3 (-2 +
\ol{c})^3 (273 - 36 c^2 - 36
\ol{c}^2 + 16 c^2
\ol{c}^2)}{(2 + c)^2 (2 +
\ol{c})^2}x^{22} +4 \frac{(-2 + c)^2 (-2 +
\ol{c})^2 (455 - 116 c^2 + 4 c^4 - 116
\ol{c}^2 + 44 c^2
\ol{c}^2 + 4
\ol{c}^4)}{(2 + c)^2 (2 +
\ol{c})^2}x^{24}$

$+4\frac{ (-2 + c)^2 (-2 +
\ol{c})^2 (-35 + 4 c^2 + 4
\ol{c}^2)}{(2 + c) (2 +
\ol{c})}x^{26} -24\frac{ (-2 + c) (-2 +
\ol{c} )(-5 + c^2 +
\ol{c}^2)}{(2 + c) (2 +
\ol{c})}x^{28}-4 (-2 + c) (-2 +
\ol{c})x^{30}+x^{32},$}\\
 respectively. 

Each conjugate of an algebraic number in absolute value can be bounded by
$|\alpha^{\prime}|\leq\max\left\{  1,n|a^{\prime}|\right\}  $, where
$|a^{\prime}|=\max\left\{  |a_{0}|,\ldots,|a_{n-1}|\right\}  $ and
$a_{0},\ldots,a_{n-1}$ are coefficients of the related monic polynomial
$\prod_{i=1}^{k}(x-\alpha^{(i)})$. Hence,
$$
h(\alpha)\leq\log(\max\{1,n|a^{\prime}|\}).$$
It is easy to see, that  each coefficient of the polynomials $p_1(x)$ and $p_2(x)$  can be bounded (in the absolute value) by $6 +16| c| + 8| c|^2$. All  coefficients of  $p_3(x)$ can be bounded by $(| c | +2)^{16} (6435 + 3168 | c |^2 +112 | c |^4 +128 | c |^6)$ - a very rough bound.
 So, 
$$h(\eta),\, h(\vartheta)\leq \log(8(6 +16| c| + 8| c|^2))<28.12,$$
$$h(\xi)\leq\log\left(32(| c | +2)^{16} (6435 + 3168 | c |^2 +112 | c |^4 +128 | c |^6)\right)<271.82,$$
and obviously $h'(\eta)$, $h'(\vartheta)$ and $h'(\xi)$ are less than the values given above. Finally, since $d\leq32\cdot8\cdot 8$ we have
$$-\log |\Lambda| \leq18\cdot4!\cdot3^{4}(32\cdot2048)^{5}28.12^2\cdot 271.82 \cdot\log(2\cdot3\cdot2048)\log l<8.6\cdot10^{34}\log l,$$
where $l=\max\{m,n\}$. If $l=m$, applying $|\Lambda|<3^{-m}$    to the previous inequality, we get
$$ \frac{m}{\log m}<7.8\cdot10^{34}$$
which does not hold for $m\geq 6.7\cdot10^{36}$. Therefore, we solve 
$$ |\Lambda|=|\log\eta|\left|n
-m\frac{\log\vartheta}{\log\eta} +\frac{\log\xi}{\log\eta}\right|<3^{-m}, \ m< 6.7\cdot10^{36},$$
ie.
\begin{equation}\label{333}
|m\theta-n +\gamma|<\delta\cdot 3^{-m} \end{equation}
where $\theta=\frac{\log\vartheta}{\log\eta}$, $\gamma=-\frac{\log\xi}{\log\eta}$ and $\delta=\frac{1}{|\log\eta|}$. 

If $l=n,$ applying $|\Lambda|<1.55^{-n}$ to the previous inequality, we get
$$
\frac{n}{\log n}<2\cdot 10^{35}$$ 
which does not hold for $n\geq 1.715\cdot10^{37}$. Therefore, we solve
$$
|\Lambda|=|\log\vartheta|\left\vert m-n\frac{\log\eta}{\log\vartheta}%
+\frac{\log\xi}{\log\vartheta}\right\vert <\left(  1.55\right)  ^{-n}%
,\ \ \ n<1.715\cdot10^{37}$$
ie.
\begin{equation}\label{155}
|n\theta^{\prime}-n+\gamma^{\prime}|<\delta^{\prime}\cdot1.55^{-n}
\end{equation}
where $\theta^{\prime}=\frac{\log\eta}{\log\vartheta}$, $\gamma^{\prime}%
=\frac{\log\xi}{\log\vartheta}$ and $\delta^{\prime}=\frac{1}{|\log\vartheta|}$.\\

Now we will apply the reduction method similar to one described in \cite{DU-reduk}.
\begin{lemma}[{\cite[Lemma 4a]{DU-reduk}}]
Let $M$ be a positive integer and let $p/q$ be a convergent
of the continued fraction expansion of $\theta$ such that $q > 6M$. Furthermore,
let $\epsilon = \|\gamma q\| - M \|\theta q\|$, where $\| \cdot \|$ denotes the distance from the nearest
integer. If $\epsilon > 0$, then the inequality
$$ |m\theta - n + \gamma| < \delta a^{-m}$$
has no integer solutions $m$ and $n$ such that $\log(\delta q/\epsilon)/ \log a \leq m \leq M$.
\end{lemma}

Since our bound for the absolute value of $c$ is very huge (almost $160\, 000$), we  perform reductions only
for $|c|\leq200$, $c\in\Z_M$. We obtained that (\ref{333}) and (\ref{155})  has no integer solutions for $m\geq n>22$
and $n\geq m >55$, respectively. The reason for not achieving a better bound for $m$ and $n$
is because $\theta$ and $\theta'$ are very close to $1$ and hence  their first convergent is too large. (For instance,
for $c=1+66\sqrt{-2}$ related $\theta<1.000044$ and the denominator of the  first convergent is $q=22\,788$.) Finally, 
we showed that the equations $u_m=\pm u_n'$ for $1\leq m,n\leq 55$ have no solutions in $\Z_M$ except $c=0,\pm 1,\pm 2$.
(Note that  according to (\ref{e12}) and  (\ref{e12}), $u_k$ and $u_k'$ are
$k$-th degree polynomials in the variable $c$. So, solving $u_m=\pm u_n'$ reduces to finding  roots of certain polynomials in $\Z_{M}$.)

\begin{proposition}
\label{PropRP}If $|c|\leq 200$ and $c\not\in S_c$,  then all non-equivalent generators of power integral
basis of $\mathcal{O=}\mathbb{Z}_{M}\left[  \xi\right]  $ over $\mathbb{Z}%
_{M}$ are $\alpha=\xi,$\ $2\xi-2c\xi^{2}+\xi^{3}$.
\end{proposition}

\section{On the case $c\in S_{c}$}

\label{s.Sc}

So far, we have observed the case when the parameter $c\notin S_{c},$ where
the set $S_{{\small c}}$ is given by (\ref{SC1}). Note, that if $c\in
S_{{\small c}},$ then on at least one of the equation of the system (\ref{e6})
and (\ref{e7}) we can not apply Lemma \ref{l.1}. Indeed, in these cases, there
are additional classes of solutions of the equation (\ref{e6}) or (\ref{e7}),
or there exists only finitely many solutions of those equations. Also note, if
$\left(  p,q\right)  =\left(  a,b\right)  \ $is a solution of the Thue
equation (\ref{e5}) for $c=c_{0},$ then $\left(  p,q\right)  =\left(
b,a\right)  \ \ $is a solution of this equation for $c=-c_{0}.$ Therefore, it
is enough to observe only $c^{\text{'}}$s from the set $S_{{\small c}}$ with
$Re(c)\geq0.$ Furthermore, all $c\in S_{{\small c}}$ are from only one
imaginary quadratic field except $c=\pm1\ $that belong to each field
$M=\mathbb{Q}\left(  \sqrt{-D}\right)$. Thus, for each $c\in S_{{\small c}}%
$, $c\neq\pm1,$ we have to find additional classes of solutions of the
equation (\ref{e6}) or (\ref{e7}) (see \cite{FJ}) and repeat the entire
procedure from previous sections. This situation is much simpler because we
have a specific value of $c$ and each $c$ is from exactly one field. On the
other hand, we need to find intersections of at least four recursive series.

For $c=1$ Thue equation (\ref{e5}) have the form
\begin{equation}
p^{4}-2p^{3}q+2p^{2}q^{2}+2pq^{3}+q^{4}=\mu,\label{t1}%
\end{equation}
and the related system is%
\[
V^{2}-3U^{2}=-2\mu,\text{ \ \ }U^{2}+Z^{2}=2\mu.
\]
By Lemma \ref{l.1}, solutions of the first equation are $\left(  V,U\right)
=\left(  \pm v_{m},\pm u_{m}\right)  ,$ where sequences $\left(  v_{m}\right)
$ and $\left(  u_{m}\right)  $ are given by
\begin{align*}
v_{0}=  &  \epsilon, & v_{1}=  &  5\epsilon, & v_{m+2}  &  =4v_{m+1}%
-v_{m},\;\;m\geq0,\\
u_{0}=  &  \epsilon, & u_{1}=  &  3\epsilon, & u_{m+2}  &  =4u_{m+1}%
-u_{m},\;\;m\geq0,
\end{align*}
where $\epsilon=1,$\textit{\ }$i,$\textit{\ }$\omega,$\textit{\ }$\omega^{2}%
$\textit{\ }corresponds to $\mu=1,-1,\omega,\omega^{2},$ respectively.
Therefore, we have to observe the equation
\begin{equation}
U^{2}+Z^{2}=2\mu,\text{ \ \ for \ }\mu\in\left\{  1,-1,\omega,\omega
^{2}\right\}  \cap\mathbb{Q}\left(  \sqrt{-D}\right)  .\label{ec1}%
\end{equation}
If $D=1\mathbf{,}$ then $\mu=1,-1$. In this case, since $-1$ is a square in
$\mathbb{Z}_{M},$ the left side of (\ref{ec1}) can be factorized as
\[
U^{2}+Z^{2}=U^{2}-\left(  -1\right)  Z^{2}=\left(  U-iZ\right)  \left(
U+iZ\right)  =2\mu,\text{ \ }\mu=1,-1.
\]
This implies that the equation (\ref{ec1}) has only finitely many solutions
\[
\left(  U,Z\right)  =\left(  \pm1,\pm1\right)  ,\text{ if \ }\mu=1\text{ \ and
\ \ }\left(  U,Z\right)  =\left(  \pm i,\pm i\right)  ,\text{ if }%
\mu=-1\text{,}%
\]
which again implies that all solutions of system of relative Pellian equations
are given by\textit{\ }
\begin{align*}
\left(  U,V,Z\right)   &  =\left(  \pm1,\pm1,\pm1\right)  ,\text{ if \ }%
\mu=1\\
\left(  U,V,Z\right)   &  =\left(  \pm i,\pm i,\pm i\right)  ,\text{ if \ }%
\mu=-1.
\end{align*}
Hence, if $\mu=1$, then the solutions of the corresponding Thue equation are
\[
(p,q)\in\{(0,\pm1),(\pm1,0),(0,\pm i),(\pm i,0)\}
\]
and if $\mu=-1$, then there are no solutions.

Note, that for $c=1$ the corresponding Thue equation (\ref{t1}) can be
transformed into equation%
\begin{equation}
X^{2}+3Y^{2}=\mu\label{j3}%
\end{equation}
by putting $X=\pm\left(  p^{2}-pq-q^{2}\right)  $ \ and\ \ $Y=\pm pq.$ The
equation (\ref{j3}) has infinitely many solutions in all rings $\mathbb{Z}%
_{M},$ except in the ring of integers of the field $\mathbb{Q}\left(
\sqrt{-3}\right)  $ because $-3$ is a square in the related ring. In that case
equation (\ref{j3}) can be factorized as
\[
X^{2}+3Y^{2}=\left(  X-\sqrt{-3}Y\right)  \left(  X+\sqrt{-3}Y\right)  =\mu,
\]
where $\mu=1,\omega,\omega^{2}$. This implies that the equation (\ref{j3}) has
only finitely many solutions%
\[
\left(  X,Y\right)  \in\left\{  \left(  \pm1,0\right)  ,\left(  \pm
\omega,0\right)  ,\left(  \pm\omega^{2},0\right)  \right\}  .
\]
Since $Y=\pm pq=0$ for each solution from above, we conclude that all
solutions Thue equation (\ref{t1}) are
\begin{align*}
(p,q) &  \in\{(0,\pm1),(\pm1,0)\}\text{ \ if \ }\mu=1,\\
(p,q) &  \in\{(0,\pm\omega),(\pm\omega,0)\}\text{ \ if \ }\mu=\omega,\\
(p,q) &  \in\{(0,\pm\omega^{2}),(\pm\omega^{2},0)\}\text{ \ if \ }\mu
=\omega^{2}.
\end{align*}
In the ring of integers $\mathbb{Z}_{M}$ of the field $M=\mathbb{Q}\left(
\sqrt{-D}\right)  ,$ where $D\neq1,3\ $for $c=1,$ we have to find all
solutions of the equation
\begin{equation}
U^{2}+Z^{2}=2. \label{2}%
\end{equation}
In this case the equation (\ref{2}) has infinitely many solutions and the form
of these solutions depend on $D.$

Therefore, we have proved:
\begin{proposition}
\label{PropC1}Let  $M=\mathbb{Q}\left(  \sqrt{-D}\right)$,  where $D=1,3$. If 
$c=1$ or $c=-1$, then non-equivalent generators of power integral basis of
$\mathcal{O=}\mathbb{Z}_{M}\left[  \xi\right]  $ over $\mathbb{Z}_{M}$ are
given by $\alpha=\xi,$ \ $2\xi-2\xi^{2}+\xi^{3}$ or $\alpha=\xi,$
\ $2\xi+2\xi^{2}+\xi^{3}$, respectively.
\end{proposition}

\section{On elements with the absolute index $1$}

\label{s.AI}Let $\mathbb{Q}\subset M\subset K$ be number fields with
$m=[M:\mathbb{Q}]$ and $k=[K:M]$. Let $\mathcal{O}$ be either the ring of
integers $\mathbb{Z}_{K}\ $of $K$ or an order of $\mathbb{Z}_{K}$. Denote
$D_{\mathcal{O}}$ and $D_{M}$ the discriminant of $\mathcal{O}$\ and subfield
$M$, respectively. Also, denote by $\gamma^{(i)}$ the conjugates of any
$\gamma\in M$ ($i=1,\ldots,m$). Let $\delta^{(i,j)}$ be the images of
$\delta\in K$ under the automorphisms of $K$ leaving the conjugate field
$M^{(i)}$ elementwise fixed ($j=1,\ldots,k$).

According to \cite{GRS} for any primitive element $\alpha\in\mathcal{O}$ we
have%
\begin{equation}
I_{\mathcal{O}}(\alpha)=\left[  \mathcal{O}^{+}:\mathbb{Z}[\alpha]^{+}\right]
=\left[  \mathcal{O}^{+}:\mathbb{Z}_{M}[\alpha]^{+}\right]  \cdot\left[
\mathbb{Z}_{M}[\alpha]^{+}:\mathbb{Z}[\alpha]^{+}\right]  . \label{ind}%
\end{equation}
The first factor we call the \textit{relative index} of $\alpha$ and we have%
\[
I_{\mathcal{O}/M}(\alpha)=\left[  \mathcal{O}^{+}:\mathbb{Z}_{M}[\alpha
]^{+}\right]  =
\]%
\begin{equation}
=\frac{1}{\sqrt{|N_{M/\mathbb{Q}}(D_{\mathcal{O}/M})|}}\cdot\prod_{i=1}%
^{m}\;\;\prod_{1\leq j_{1}<j_{2}\leq k}\left\vert \alpha^{(i,j_{1})}%
-\alpha^{(i,j_{2})}\right\vert \label{relind}%
\end{equation}
where $D_{\mathcal{O}/M}$ is relative discriminant of $\mathcal{O}$ over $M$.
For the second factor we have
\[
J(\alpha)=\left[  \mathbb{Z}_{M}[\alpha]^{+}:\mathbb{Z}[\alpha]^{+}\right]  =
\]%
\begin{equation}
=\frac{1}{\sqrt{|D_{M}|}^{[K:M]}}\cdot\prod_{1\leq i_{1}<i_{2}\leq m}%
\;\;\prod_{j_{1}=1}^{k}\prod_{j_{2}=1}^{k}\left\vert \alpha^{(i_{1},j_{1}%
)}-\alpha^{(i_{2},j_{2})}\right\vert . \label{ind2}%
\end{equation}
Generators $\alpha_{0}$ of relative power integral bases of $\mathcal{O}$ over
$M$ have relative index $I_{\mathcal{O}/M}(\alpha_{0})=1$. The elements%
\begin{equation}
\alpha=A+\varepsilon\cdot\alpha_{0}, \label{eps}%
\end{equation}
(where $\varepsilon$ is a unit in $M$ and $A\in\mathbb{Z}_{M}$) have the same
relative index, and are called \textit{equivalent} with $\alpha_{0}$ over $M$.
Equivalently, all elements $\alpha\in\mathcal{O}$ generating a power integral
basis of $\mathcal{O}$ (over $\mathbb{Q}$), that is having $I_{\mathcal{O}%
}(\alpha)=1,$ must be of the form (\ref{eps}), where $\alpha_{0}$ has relative
index $I_{\mathcal{O}/M}(\alpha_{0})=1$. In order that $\alpha$ generates a
power integral basis of $\mathcal{O}$ we must also have $J(\alpha)=1$.
Therefore for each $\alpha_{0}\in\mathcal{O}$ with relative index
$I_{\mathcal{O}/M}(\alpha_{0})=1,$ we have to determine the unit
$\varepsilon\in M$ and $A\in\mathbb{Z}_{M}$ such that $J(\alpha)=1$.

We consider the octic field $K_{c}=\mathbb{Q}(\xi),$ where $\xi$ is a root of
the polynomial $f(t)=t^{4}-2ct^{3}+2t^{2}+2ct+1,$ where $c\in\mathbb{Z}%
_{M}\backslash\left\{  0,\pm2\right\}  $, $M=\mathbb{Q}(\sqrt{-D})\ $and $D$
is a squarefree positive integer. Therefore, $m=[M:\mathbb{Q}]=2$ and $K_{c}$
is an extension of $M$ of degree $k=[K_{c}:M]=4.$

We have proved that all generators of relative power integral bases of
$\mathcal{O=}\mathbb{Z}_{M}\left[  \xi\right]  $ over $M$ are given by
\[
\alpha_{1}=\xi,\;\;\;\alpha_{2}=2\xi-2c\xi^{2}+\xi^{3},
\]
in the cases given in Theorem \ref{tm:main}. Also, according to Remark
\ref{Rem1}, $\alpha_{1}\ $and $\alpha_{2}$ are the generators of relative
power integral bases for all $c\in\mathbb{Z}_{M}\backslash\left\{
0,\pm2\right\}  .$\newline

\begin{proof}
[Proof of Theorem \ref{tm:IG}] Taking $\alpha
_{0}=\alpha_{1},$ $\alpha_{2}$ we calculate $J(\alpha)$ with the $\alpha$ in
(\ref{eps}). For $-D\equiv2,3\;(\operatorname{mod}4)$ an integral basis of $M$
is given by $\{1,\omega\}$ with $\omega=\sqrt{-D}$. We have
\[
\sqrt{|D_{M}|}^{[K:M]}=16D^{2}.
\]
We set $c=p+q\omega$ with integer parameters $p,q$. Let $A=a+b\omega$ with
$a,b\in\mathbb{Z}$. Note that the product (\ref{ind2}) in $J(\alpha)$ does not
depend on $a$. We have $\varepsilon=\pm1$ and for $-D=-1$ we also have
$\varepsilon=\pm i$. The product
\begin{equation}
\prod_{j_{1}=1}^{4}\prod_{j_{2}=1}^{4}\left\vert \alpha^{(1,j_{1})}%
-\alpha^{(2,j_{2})}\right\vert \label{jjj}%
\end{equation}
is of degree $16$, depending on $D,p,q\ $and $b$. We calculated this product
by Maple using symmetric polynomials. The result is a very complicated
polynomial with integer coefficients of the above variables. We found that in
each case the above product was divisible by $4096D^{2}$. Therefore dividing
it by $16D^{2}$ the $J(\alpha)$ is divisible by $256$. This implies that we
cannot have $J(\alpha)=1$, therefore we cannot have $I_{\mathcal{O}}%
(\alpha)=1$.
\end{proof}
\\

\textbf{Computational aspects}

It was very difficult to perform the calculation of the product (\ref{jjj}).
We had to do it in several steps making simplifications by using symmetric
polynomials in each step. Even so, this calculation has reached the limits of
the capacities of Maple. We were not able to perform this calculation for
$-D\equiv1\;(\operatorname{mod}4)$.\medskip

\textbf{Acknowledgements.} The authors would like to thank Professor Andrej
Dujella for helpful suggestions. The idea for Theorem \ref{tm:IG} and the proof of it 
due to Professor Istv\'{a}n Ga\'{a}l.

{\footnotesize \textsc{B. Jadrijevi\' c, Department of Mathematics, University of Split, R. Bo\v skovi\' ca 33, 21000 Split, Croatia,}  \texttt{borka@pmfst.hr}\\

 \textsc{Z. Franu\v si\' c, Department of Mathematics, Faculty of Science, University of Zagreb,  Bijeni\v cka
30, 10000 Zagreb, Croatia,}  \texttt{fran@math.hr} }

{\footnotesize \emph{E-mail address}: \texttt{borka@pmfst.hr} }

\end{document}